\documentclass{amsart}

\usepackage[cp1251]{inputenc}
\usepackage[usenames,dvipsnames]{xcolor}
\usepackage{fullpage}
\usepackage[upright]{fourier}
\usepackage{tikz,tikz-cd}
\usetikzlibrary{arrows,arrows.meta}
\usepackage{amssymb,amsmath,amsthm,eucal,amscd}
\usepackage{amsfonts}
\usepackage[figurename=Fig.]{caption}
\usepackage[position=top]{subfig}
\usepackage{amsrefs}

\newcommand{\Z}{\mathbb Z}

\newcommand{\R}{\mathbb R}
\newcommand{\C}{\mathbb C}

\newcommand{\T}{\mathbb T}
\newcommand{\F}{\mathbb F}
\newcommand{\A}{\mathbb A}

\renewcommand{\P}{\mathbb P}

\newcommand{\mc}{\mathcal}

\newcommand{\lb}{\lbrace}
\newcommand{\rb}{\rbrace}
\newcommand{\la}{\langle}
\newcommand{\ra}{\rangle}
\renewcommand{\phi}{\varphi}
\renewcommand{\mid}{\,;\,}

\DeclareMathOperator{\rk}{rk}

\DeclareMathOperator{\cut}{cut}
\DeclareMathOperator{\Ann}{Ann}

\DeclareMathOperator{\Bl}{Bl}
\DeclareMathOperator{\Aut}{Aut}
\DeclareMathOperator{\Hom}{Hom}

\DeclareMathOperator{\Ker}{ker}
\DeclareMathOperator{\Pic}{Pic}
\DeclareMathOperator{\Id}{Id}
\DeclareMathOperator{\Mat}{Mat}

\tikzcdset{arrow style=tikz, diagrams={>=stealth}}

\renewcommand{\leq}{\leqslant}
\renewcommand{\geq}{\geqslant}

\theoremstyle{plain}
\newtheorem{thm}{Theorem}[section]
\newtheorem{lm}[thm]{Lemma}
\newtheorem{cor}[thm]{Corollary}
\newtheorem{pr}[thm]{Proposition}

\theoremstyle{remark}
\newtheorem{rem}[thm]{Remark}
\newtheorem{ex}[thm]{Example}
\theoremstyle{definition}
\newtheorem{defn}[thm]{Definition}
\newtheorem{constr}[thm]{Construction}
\newtheorem{asm}[thm]{Assumption}

\begin{document}

\title{Cohomology rings and algebraic torus actions on hypersurfaces in the product of projective spaces and bounded flag varieties
}

\author{Grigory Solomadin}
\address[G.\,Solomadin]{Laboratory of algebraic topology and its applications, Faculty of computer science, National Research University Higher School of Economics, Russian Federation}
\email{grigory.solomadin@gmail.com}

\begin{abstract}
In this paper, for any Milnor hypersurface we find the largest dimension of effective algebraic torus actions on it. The proof of the corresponding theorem is based on the computation of the automorphism group for any Milnor hypersurface. We find all generalised Buchstaber-Ray and Ray hypersurfaces that are toric varieties. We compute the Betti numbers of these hypersurfaces and describe their integral singular cohomology rings in terms of the cohomology of the corresponding ambient varieties.
\end{abstract}

\keywords{Toric varieties: automorphisms of algebraic varieties, torus actions, blow-ups, fiber bundles, hypergraphs}
\subjclass[2020]{Primary: 14L30, 53D20; secondary: 14M25, 14J50, 94C15}
\thanks{The publication has been prepared with the support of ``RUDN University Program 5--100'' program. The reported study was funded by RFBR, project number $20-01-00675\ A$. The reported study was funded by the grant of ``Young mathematics of Russia'' foundation.}

\maketitle

\section{Introduction}

In the present paper, we study effective algebraic torus actions on the particular collections of nonsingular complex algebraic hypersurfaces, namely, $H_{i,j}$, $BR_{i,j}$ and $R_{i,j}$ in $\P^{i}\times \P^{j}$, $BF_{i}\times \P^j$ and $BF_{i}\times BF_j$, respectively, for any nonnegative integers $i,j$. Here the $n$-dimensional varieties $\P^n$ and $BF_{n}$ are a complex projective space and a bounded flag variety \cite{bu-ra-98'}, respectively.

For any integers $i, j \geq 0$ the transverse intersection $H_{i,j}$ of the Segre embedding image of $\P^i\times\P^j$ to \\
$\P^{(i+1)(j+1)-1}$ with a generic hyperplane is called a \emph{Milnor hypersurface}. In particular, $H_{i,j}$ is a hypersurface in $\P^i\times\P^j$ of bidegree $(1,1)$. The hypersurface $BR_{i,j}$ was defined as a toric variety in \cite{bu-ra-98} for any integers $0\leq i\leq j$. Following the definition of the hypersurface $R_{i,j}$ given in \cite{ra-86} for any integers $i, j \geq 0$ by Ray, we call it a \emph{Ray hypersurface}.

Recall that a normal algebraic variety $X$ over $\C$ containing an algebraic torus $\T$ as a dense open orbit is called a \emph{toric variety} if the action of $\T$ on itself extends to a regular action on $X$. The motivation for our study stems from the question raised in \cite{so-17}: is $R_{i,j}$ a toric variety? A positive answer to this question leads to the short proof of one theorem from algebraic topology, as described in \cite{so-17}.

For any integers $i, j \geq 0$ it was shown in \cite{bu-ra-98} that the variety $H_{i,j}$ is a toric variety iff $\min\lbrace i,j \rbrace\leq 1$. Demazure's result \cite{dem-70} allows to describe the automorphism group of any Milnor hypersurface that is a toric variety. We remark that the automorphism group of $H_{1,3}$ was described explicitly in \cite[Lemma 4.5]{ch-pr-sh-19}. We compute the automorphism group of $H_{i,j}$ for arbitrary integers $i,j\geq 0$. The computation is based on the well-known sheaf-theoretic argument for projective Fano varieties. We deduce the first main result of this paper from this computation.

\begin{thm}\label{thm:milnordim}
The largest dimension for algebraic torus actions on the Milnor hypersurface $H_{i,j}$ is equal to $\max{\lb i,j\rb}$ for any integers $i,j\geq 0$.
\end{thm}

We provide a natural definition of the variety $BR_{i,j}$ as a hypersurface in $BF_{i}\times \P^j$ for all integers $i,j\geq 0$. Taking into account that $BR_{i,j}$ is isomorphic to the variety from \cite{bu-ra-98} for any integers $i,j\geq 0$ such that $i\leq j$, we call the hypersurface $BR_{i,j}$ a \emph{generalised Buchstaber-Ray hypersurface}. The following two theorems represent main results of this paper, in addition to Theorem \ref{thm:milnordim}.

\begin{thm}\label{thm:brijnottoric}
The hypersurface $BR_{i,j}$ is a toric variety iff \/ $0\leq i\leq j$ or $j=0,1$.
\end{thm}

\begin{thm}\label{thm:rijnontoric}
The hypersurface $R_{i,j}$ is a toric variety iff \/ $\min{\lb i,j\rb}=0,1$ or $i=j=2$.
\end{thm}

Theorem \ref{thm:rijnontoric} provides a complete answer to the problem discussed in \cite{so-17}. In order to prove Theorem \ref{thm:brijnottoric}, for any integers $i,j$ such that $0\leq i\leq j$ or $j=0,1$, we define the algebraic torus action on $BR_{i,j}$ endowing it with the structure of a toric variety. For any integers $i,j\geq 0$ we define the effective action of the algebraic torus $(\C^{\times})^{\max\lbrace i,j\rbrace}$ on $BR_{i,j}$. This action corresponds to the $\max\lbrace i,j\rbrace$-dimensional algebraic subtorus $\T$ in the connected component $\Aut^0 BR_{i,j}$ of the automorphism group $\Aut BR_{i,j}$ of $BR_{i,j}$. Let $i,j\geq 0$ be any integers that do not satisfy the condition of Theorem \ref{thm:brijnottoric}. Let $\T'$ be any maximal algebraic torus in $\Aut^0 BR_{i,j}$ such that $\T\subseteq \T'$. All maximal algebraic tori of the algebraic group $\Aut^0 BR_{i,j}$ are conjugate to each other. We prove that $BR_{i,j}$ with $\T'$-action is not a toric variety by using a combination of methods from \cite{gu-za-01g}, \cite{ba-14} and \cite{ta-04}. These two facts together imply that $BR_{i,j}$ is not a toric variety (for these particular values of $i,j$). We prove Theorem \ref{thm:rijnontoric} by following a similar approach.

In addition, for all integers $i,j\geq 0$ we compute the Betti numbers of the hypersurfaces $BR_{i,j}$ and $R_{i,j}$, and relate their integral singular cohomology rings to the cohomology rings of $BF_{i}\times \P^j$ and $BF_{i}\times BF_j$, respectively. Namely, we prove that the morphism of the respective integral cohomology rings, induced by the embedding of any hypersurface considered above to the ambient space, is onto, and describe its kernel.

The paper is organised as follows. In Section \ref{sec:autom}, the automorphism group of any Milnor hypersurface is computed and the proof of Theorem \ref{thm:milnordim} is provided. In Section \ref{sec:def}, we define generalised Buchstaber-Ray and Ray hypersurfaces. In Section \ref{sec:torus_act}, we define a certain class of algebraic torus actions on any nonsingular complex manifold. We assign the hypergraph equipped with additional structures to any action from this class. These structures generalise the notion of an axial function and a connection from GKM-theory (see \cite{gu-za-01}) to the case of a hypergraph. In Section \ref{sec:appl}, the proofs of Theorems \ref{thm:brijnottoric} and \ref{thm:rijnontoric} are given. In Appendix \ref{sec:blowup}, we describe the generalised Buchstaber-Ray and Ray hypersurfaces in terms of consecutive blow-ups along smooth subvarieties as well as in terms of algebraic fiber bundles.  In Appendix \ref{sec:coh}, we study the integral singular cohomology rings of generalised Buchstaber-Ray and Ray hypersurfaces, and compute the respective Betti numbers by utilizing the results from Appendix \ref{sec:blowup}.

\section{The automorphism group of a Milnor hypersurface}\label{sec:autom}

Unless explicitly stated otherwise, in the sequel an \emph{algebraic variety} (or, in short, a \emph{variety}) is defined as a separated reduced irreducible scheme of finite type over $\C$. A \emph{hypersurface} in a variety is a subvariety of codimension $1$. An \emph{algebraic fiber bundle} is a locally trivial algebraic fiber bundle in the Zariski topology. A \emph{holomorphic fiber bundle} is a locally trivial complex-analytical fiber bundle over a complex manifold. We call any toric variety $X^n$ that is an algebraic fiber bundle $\pi\colon X\to B$ a \emph{toric fiber bundle}, if the base $B$ and the fiber $F$ are toric varieties and the projection $\pi$ is equivariant with respect to the given algebraic torus actions on $X$ and $B$. A \emph{fiber bundle} is a locally trivial topological fiber bundle. Occasionally, we call a fiber bundle with a particular structure (topological, holomorphic, algebraic, toric) with fiber $F$ an \emph{$F$-bundle}. We indicate the complex dimension $\dim X=n$ of an algebraic variety (or complex manifold) $X$ by writing $X^n$. We put $\dim \varnothing:=-1$.

In this paper, we repeatedly use the well-known bijective correspondence between (Cartier) divisors on a nonsingular algebraic variety $X$ and algebraic line bundles over $X$ (\cite[p.144]{ha-77}). This correspondence respects the equivalence relations of linear equivalence on divisors and of algebraic isomorphism on line bundles. Another variant of this correspondence takes place for complex manifolds and holomorphic line bundles, with appropriately defined equivalence relations in the holomorphic setting. For more details, see \cite[Chapter 1, \S1]{gr-ha-78}.

We denote by $\xi^{\vee}$ the dual vector bundle to any vector bundle $\xi$ (with a particular structure). We slightly abuse the notation and denote the pull-backs of all vector bundles $\xi\to X$, $\eta\to Y$ under the natural projections $X\times Y\to X$ and $X\times Y\to Y$ of varieties by $\xi$ and $\eta'$, respectively.

We consider the set $\Aut X$ of all automorphisms of any algebraic variety $X$ as an abstract group with the natural group operation.

\begin{defn}\label{defn:aut0}
The group $\Aut X$ is called the \emph{automorphism group} of an algebraic variety $X$. The \emph{connected component} $\Aut^0 X$ of the group $\Aut X$ is the subgroup of automorphisms that occur as a member of a family $\lbrace\phi_{b}\rbrace_{b\in B}$ such that $B$ is an irreducible rational curve, the natural map $B\times X\to X$ defined by $(b,x)\mapsto \phi_{b}(x)$ is a morphism, and $\phi_{b_0}=\Id_{X}$ is the identity for some $b_0\in B$.
\end{defn}

It follows from the Definition \ref{defn:aut0} that for any algebraic torus $\T$ acting on $X$ its image under the natural embedding to $\Aut X$ is contained in $\Aut^0 X$ \cite[Lemma 1.4, p. 1715]{ar-ba-13}.

\begin{pr}\cite[Corollary 1, p.31]{ra-64}\label{pr:autalg}
Let \/ $X^{n}$ be a nonsingular complete variety. Then $\Aut^0 X$ is an algebraic group.
\end{pr}

\begin{pr}\label{pr:auttoric}\cite{dem-70}
Let $X^{n}$ be a nonsingular projective toric variety. Then $\Aut X$ is an algebraic group of rank $n$.
\end{pr}

\begin{cor}\label{cor:extassum}
Let \/ $X^{n}$ be a nonsingular projective variety. Let $k$ be the rank of \/ $\Aut^0 X$. For any integer $r\geq 0$ and any effective action of \/ $\T^{r}:=(\C^{\times})^r$ by automorphisms on $X^{n}$ the following holds.

$(i)$ One has $r\leq k$, and there exists an extension of \/ $\T^{r}$-action on \/ $X^n$ to an effective action of \/ $\T^k$ on \/ $X^{n}$;

$(ii)$ Any two effective $\T^k$-actions by automorphisms on \/ $X^{n}$ are equivariantly isomorphic;

$(iii)$ If \/ $X^n$ is a toric variety, then the action of any maximal torus in \/ $\Aut^0 X$ on \/ $X^{n}$ endows \/ $X^n$ with the structure of a toric variety.
\end{cor}
\begin{proof}
Claims $(i)$, $(ii)$ follow from the theorem about conjugacy of all maximal algebraic tori in any algebraic group (\cite[p.119]{vi-on-90}) and Proposition \ref{pr:autalg}. Claim $(iii)$ follows from Proposition \ref{pr:auttoric} and $(ii)$.
\end{proof}

\begin{defn}\label{defn:Milnor}
For for any integers $i,j\geq 0$ the nonsingular hypersurface $H_{i,j}$ in $\P^i\times\P^j$ given by the equation
\begin{equation}\label{eq:hij_def_eq}
\sum_{k=0}^{\min{\lb i,j\rb}} z_{k}w_{k}=0.
\end{equation}
in the homogeneous coordinates $(z,w)=([z_0:\dots:z_{i}],[w_0:\dots:w_{j}])$ of $\P^i\times\P^j$ is called a \emph{Milnor hypersurface}. Denote by $\widehat{H}_{i,j}$ the hypersurface in $\P^i\times\P^j$ given by the equation
\begin{equation}\label{eq:hij_my_eq}
\sum_{k=0}^{\min{\lb i,j\rb}} z_{i-k}w_{j-k}=0.
\end{equation}
\end{defn}

The Milnor hypersurface $H_{i,j}$ is the divisor corresponding to the algebraic line bundle $\eta^{\vee}\otimes (\eta')^{\vee}$ over $\P^{i}\times \P^{j}$. Here $\eta$ denotes the tautological line bundle over a complex projective space.

\begin{rem}\label{eq:symmetryhij}
The suitable automorphism of $\P GL_{i+1}(\C)\times \P GL_{j+1}(\C)$ induces the isomorphism $\widehat{H}_{i,j}\simeq H_{i,j}$ of subvarieties in $\P^i\times \P^j$. The map $\P^{i}\times \P^{j}\to \P^{j}\times \P^{i}$, $(z,w)\mapsto (w,z)$, maps $H_{i,j}$ to $H_{j,i}$. Hence, $H_{i,j}\simeq H_{j,i}$.
\end{rem}
It is well known that $\Aut \P^n\simeq \P GL_{n+1}(\C)$ (\cite[Example 7.1.1, p.152]{ha-77}). It is easy to prove the following lemma.

\begin{lm}\label{lm:prod}
Let $i,j\geq 0$ be any integers. If $i\neq j$, then $\Aut (\P^i\times \P^j)\simeq \P GL_{i+1}(\C)\times \P GL_{j+1}(\C)$. One has\\ $\Aut (\P^i\times \P^i)\simeq\bigl(\P GL_{i+1}(\C)\times\P GL_{i+1}(\C)\bigr)\rtimes\Z_{2}$.
\end{lm}

We extend any automorphism of $H_{i,j}$ to the automorphism of $\P^i\times \P^j$ as follows.

\begin{lm}\label{lm:restrict}
There is the monomorphism of algebraic groups $\Aut H_{i,j}\to \Aut (\P^i\times \P^j)$. Its image consists of automorphisms of\/ $\P^i\times\P^j$ leaving \/ $H_{i,j}$ invariant.
\end{lm}
\begin{proof}
Recall that there is the standard exact sequence relating the ideal sheaf of the subvariety to the structure sheaf of the ambient variety. For the natural inclusion $\iota\colon H_{i,j}\to\P^i\times \P^j$, the corresponding exact sequence of sheaves on $\P^i\times \P^j$ is
\begin{equation}\label{eq:subvsheaf}
0\to \mc{O}_{\P^i\times\P^j}(-1,-1)\to\mc{O}_{\P^i\times\P^j}\to\iota_{*}\mc{O}_{H_{i,j}}\to 0.
\end{equation}
Twisting \eqref{eq:subvsheaf} by $\mc{O}_{\P^i\times\P^j}(1,1)$ one obtains the following exact sequence
\begin{equation}\label{eq:exseq}
0\to \mc{O}_{\P^i\times\P^j}\to\mc{O}_{\P^i\times\P^j}(1,1)\to\iota_{*}\mc{O}_{H_{i,j}}(1,1)\to 0,
\end{equation}
of sheaves. By \cite[Lemma 2.10, p.209]{ha-77}, one has
\begin{equation}\label{eq:restrsh}
H^{0}(\P^i\times \P^j\mid \iota_* \mc{O}_{H_{i,j}}(1,1))= H^{0}(H_{i,j}\mid \mc{O}_{H_{i,j}}(1,1)).
\end{equation}
It follows from the cohomological long exact sequence of \eqref{eq:exseq}, the identity $H^{1}(\P^i\times\P^j\mid \mc{O}_{\P^i\times\P^j})=0$ (which in turn follows from K\"{u}nneth's formula and the description of sheaf cohomology of $\P^n$) and \eqref{eq:restrsh} that
\begin{equation}\label{eq:episheaf}
H^{0}(\P^i\times\P^j\mid \mc{O}_{\P^i\times\P^j}(1,1))\to H^{0}(H_{i,j}\mid \mc{O}_{H_{i,j}}(1,1))
\end{equation}
is an epimorphism. It is not hard to show that the abelian group $H^2(H_{i,j};\Z)\simeq \Z^2$ is generated by the first Chern classes of the restrictions of the sheaves $\mc{O}_{\P^i\times\P^j}(0,1)$, $\mc{O}_{\P^i\times\P^j}(1,0)$ to $H_{i,j}$. Then one obtains $\Pic H_{i,j}\simeq \Z^2$ from the following part of the long exact sequence
\[
0= H^1(H_{i,j}\mid \Omega)\to H^1(H_{i,j}\mid \C^{\times})\to H^2(H_{i,j}\mid \Z)\to H^{2}(H_{i,j}\mid \Omega)=0,
\]
of the exponential sequence of sheaves, where $\Omega$ is the sheaf of germs of local holomorphic functions on $H_{i,j}$ (see \cite[p.127, \S 15.9]{hi-66}). The classes of $\iota^* \mc{O}_{\P^i\times\P^j}(0,1)$, $\iota^* \mc{O}_{\P^i\times\P^j}(1,0)$ span the semigroup of effective divisors in $\Pic H_{i,j}$. Any automorphism $\phi\in\Aut H_{i,j}$ maps effective divisors to effective. Hence, the abelian group isomorphism $\phi^*$ defines the bijective map on the basis of the semigroup of effective divisors to itself. We conclude that the homomorphism $\phi^*: \Pic H_{i,j}\to \Pic H_{i,j}$ restricts to the well-defined map on the set of generators of this semigroup, represented by $\mc{O}_{\P^i\times\P^j}(0,1)$ and $\mc{O}_{\P^i\times\P^j}(1,0)$. This map is either identity or involution. Hence, $\phi^*\mc{O}_{H_{i,j}}(1,1)\simeq \mc{O}_{H_{i,j}}(1,1)$, and $\phi^*$ acts on the sections of $\mc{O}_{H_{i,j}}(1,1)$. We lift the automorphism $\phi^*$ to an automorphism of $H^0(\P^i\times\P^j\mid \mc{O}_{\P^i\times\P^j}(1,1))$ by choosing any section of the epimorphism \eqref{eq:episheaf} of $\C$-modules. The projective embedding corresponding to the sheaf $ \mc{O}_{\P^i\times\P^j}(1,1)$ is the Segre embedding
\[
\P^i\times \P^j\to \P H^{0}(\P^i\times\P^j\mid \mc{O}_{\P^i\times\P^j}(1,1)).
\]
We conclude that the automorphism $\phi$ of $H_{i,j}$ is the restriction of an automorphism of $\P^i\times \P^j$ to $H_{i,j}$. It also remains to notice that $\phi(H_{i,j})=H_{i,j}$ is an algebraic condition on $\phi\in \Aut (\P^i\times\P^j)$.
\end{proof}

By Remark \ref{eq:symmetryhij}, one has $\Aut H_{i,j}\simeq \Aut H_{j,i}$. Without loss of generality, we compute the group $\Aut H_{i,j}$ for any integers $i,j\geq 0$ such that $i\leq j$. Let $Q_{0}:\C^{j+1}\times\C^{j+1}\to \C$ be the bilinear form on $\C^{j+1}$ given by the formula
\[
Q_{0}(z,w)=\sum_{k=0}^{j}z_k w_k,
\]
for any $z=(z_0,\dots,z_j), w=(w_0,\dots,w_j)\in\C^{j+1}$. Let $\pi:\C^{j+1}\to\C^{i+1}$ be the projection given by the formula $\pi(z):=(z_{0},\dots,z_{i})$. Define the bilinear form $Q:\C^{j+1}\times\C^{j+1}\to \C$ by the formula
\[
Q(z,w):=Q_{0}(\pi (z),w)=\sum_{k=0}^{i} z_{k}w_k.
\]

Let $A\in GL_{i+1}(\C)$, $B\in GL_{j+1}(\C)$. Define $\widetilde{A}:=\widetilde{A}(A)\in GL_{j+1}(\C)$ as
\[
\widetilde{A}=
\begin{pmatrix}
A & 0\\
0 & \Id_{j-i}
\end{pmatrix},
\]
where $\Id_{j-i}$ is the identity $(j-i)\times (j-i)$-matrix and the block structure is with respect to the decomposition
\begin{equation}\label{eq:dirsum}
\C^{j+1}=\C\la e_{0},\dots,e_{i}\ra\oplus\C\la e_{i+1},\dots,e_{j}\ra,
\end{equation}
in the basis $e_0,\dots, e_j$ of $\C^{j+1}$. The proof of the following lemma is straight-forward.

\begin{lm}\label{lm:cond}
Let $A\in GL_{i+1}(\C)$, $B\in GL_{j+1}(\C)$. Suppose that for any $z,w\in \C^{j+1}$ such that $Q(z,w)=0$, one has $Q(\widetilde{A}z,Bw)=0$. Then the identity
\[
B=
\begin{pmatrix}
(A^{t})^{-1} & C\\
0 & B'
\end{pmatrix},
\]
holds for some $B'=B'(B)\in GL_{j-i}(\C)$ and some $C=C(A,B)\in \Mat_{i+1,j-i}(\C)$. The class $[B']\in \P GL_{j-i}(\C)$ is uniquely defined by the class $[B]\in \P GL_{j+1}(\C)$.
\end{lm}

For all $0<i<j$, let
\begin{equation}\label{eq:eijsubg}
E_{i,j}:=\bigg\lb ([A],[B])\in \P GL_{i+1}(\C)\times \P GL_{j+1}(\C)\bigg|\ B=
\begin{pmatrix}
(A^{t})^{-1} & 0\\
0 & B'
\end{pmatrix},\ B'\in GL_{j-i}(\C)\bigg\rb,
\end{equation}
be the subgroup of $\Aut H_{i,j}$. (This is a subgroup because the identity $((A_1 A_2)^t)^{-1}=((A_1)^t)^{-1}((A_2)^t)^{-1}$ holds for any $A_1, A_2\in GL_{i+1}(\C)$. The inclusion $E_{i,j}\subseteq \Aut H_{i,j}$ easily follows from \eqref{eq:hij_def_eq}.) The following proposition is straight-forward to prove.

\begin{pr}\label{pr:ext}
The group $E_{i,j}$ is a central extension of the following groups
\[
0\to\C^{\times}\to E_{i,j}\to \P GL_{i+1}(\C)\times\P GL_{j-i}(\C)\to 0,
\]
where the right homomorphism is given by $([A],[B])\mapsto ([A],[B'])$ in terms of \eqref{eq:eijsubg}.
\end{pr}

\begin{thm}\label{thm:authij}
Let $i,j\geq 0$ be any integers such that $i\leq j$. One has $\Aut H_{0,j}\simeq \P GL_{j}(\C)$. If \/ $0< i< j$, then $\Aut H_{i,j}\simeq \C^{(i+1)(j-i)}\rtimes E_{i,j}$. For $0< i=j$ one has $\Aut H_{i,i}\simeq \P GL_{i+1}(\C)\rtimes\Z_{2}$. In particular, $\rk \Aut H_{i,j}=j$ holds for any $0\leq i\leq j$.
\end{thm}
\begin{proof}
Since $H_{0,j}\simeq \P^{j-1}$, one has $\Aut H_{0,j}\simeq \P GL_{j}(\C)$. Now let $i>0$. We apply Lemma \ref{lm:restrict}. In the case of $i=j$, the involution $(z,w)\mapsto (w,z)$ descends from $\P^i\times \P^i$ to $H_{i,i}$. Hence, by Lemma \ref{lm:prod}, in order to prove the claim of the theorem it remains to compute the subgroup of elements in $\P GL_{i+1}(\C)\times \P GL_{j+1}(\C)$ with well-defined restrictions to $H_{i,j}$. This follows easily from Lemma \ref{lm:cond}. The proof is complete.
\end{proof}

\begin{proof}[Proof of Theorem \ref{thm:milnordim}]
Follows from Theorem \ref{thm:authij} and Corollary \ref{cor:extassum}.
\end{proof}

\begin{rem}\label{rem:pglbundle}
The quotient $GL_n(\C)\to GL_n(\C)/(\C^{\times})=\P GL_n(\C)$ by the subgroup of the diagonal matrices is a principal $\C^{\times}$-bundle. Let $\eta\to \P GL_{n}(\C)$ be the algebraic line bundle associated with it. Denote by $\eta^{\times}$ the associated $\C^{\times}$-bundle over $\P GL_{n}(\C)$ corresponding to $\eta$. In particular, the total space of the algebraic fiber bundle $\eta^{\times}$ over $\P GL_{n}(\C)$ is $GL_{n}(\C)$. The fiberwise transposed algebraic line bundle $\eta^{t}\to \P GL_{n}(\C)$ is defined in the obvious way. There is the natural isomorphism of the algebraic line bundles $\eta^t, \eta$. The group $\Pic (\P GL_n(\C))$ is isomorphic to $\Z/n\Z$ (see \cite{ba-br-65}). The first Chern class $c_{1}(\eta)$ is the generator of this cyclic group. In terms of Proposition \ref{pr:ext}, the group $E_{i,j}$ as a variety is isomorphic to the total space of the $\C^{\times}$-bundle $((\eta)^{-1}\otimes\eta')^{\times}\to \P GL_{i+1}(\C)\times\P GL_{j-i}(\C)$.
\end{rem}

Let us compute $\Aut H_{1,2}$ by applying Theorem \ref{thm:authij}.

\begin{ex}\label{ex:appl}
The algebraic line bundles $\eta,\eta^{-1}$ over $\P GL_{2}(\C)$ are isomorphic, because $\Pic \P GL_{2}(\C)=\Z/2\Z$. By Remark \ref{rem:pglbundle}, the total space of the algebraic $\C^{\times}$-bundle $\eta^\times\to \P GL_{2}(\C)$ is $GL_{2}(\C)$. We conclude that the total space of the algebraic fiber bundle $(\eta^{-1})^{\times}$ over $\P GL_{2}(\C)$ is isomorphic to $GL_{2}(\C)$. By Remark \ref{rem:pglbundle} and Theorem \ref{thm:authij} we obtain the isomorphism of algebraic groups
\begin{equation}\label{eq:h12}
\Aut H_{1,2}\simeq \C^2\rtimes (\eta^{-1})^{\times}\simeq \C^2\rtimes GL_2(\C).
\end{equation}
\end{ex}

The Milnor hypersurface $H_{1,2}$ is a toric variety \cite[pp.348--350]{bu-pa-15}. Its automorphism group can be computed by Demazure's theorem (see \cite{dem-70}, \cite[\S 3.4]{oda-88}, \cite[Excercise 4.9, p. 329]{ar-15}), and the group obtained in this way agrees with \eqref{eq:h12}. We finish this Section by defining a maximal algebraic torus in $\Aut^0 H_{i,j}$. For any integer $n\geq 0$ the formula
\begin{equation}\label{eq:pn_act}
(t_{1},\dots,t_{n})\circ z=[z_0:t_1 z_1:\dots:t_n z_n],\ (t_{1},\dots,t_{n})\in\T^n,\ z=[z_0:z_1:\dots:z_n]\in\P^n,
\end{equation}
determines the $\T^n$-action on $\P^n$. Let $i,j\geq 0$ be any integers such that $i\leq j$. Then we define the effective $\T^{j}$-action on the hypersurface $H_{i,j}$ in the homogeneous coordinates $(z,w)=([z_0:z_1:\dots:z_i],[w_0:w_1:\dots:w_j])$ of $\P^i\times \P^j$ by the formula
\begin{equation}\label{eq:hij_act}
(t_{1},\dots,t_{j})\circ(z,w)=([z_0:t_1 z_1:\dots:t_i z_i],[ w_0:t_{1}^{-1}w_1:\dots:t_j^{-1} w_j]),\ (t_{1},\dots,t_{j})\in\T^j.
\end{equation}

\section{Definitions of $BR_{i,j}$ and $R_{i,j}$}\label{sec:def}

\subsection{Generalised Buchstaber-Ray hypersurface $BR_{i,j}$}\label{ssec:brij}

Let us recall some definitions.

\begin{defn}\label{defn:bfv}[\cite{bu-ra-98}]
Let $BF_0$ be the point, and let $\beta_0:=\underline{\C}\to BF_0$ be the trivial line bundle. For any integer $n\geq 0$, let $BF_{n+1}$ be the total space of the algebraic $\P^1$-bundle $\P(\beta_{n}\oplus\underline{\C})$ associated with the algebraic vector bundle $\beta_{n}\oplus\underline{\C}$ over $BF_n$. Let $\beta_{n+1}$ be the tautological line bundle over $BF_{n+1}=\P(\beta_{n}\oplus\underline{\C})$. The variety $BF_n$ is called a \emph{bounded flag variety}. We abuse the notation slightly by defining $\beta_{k}\to BF_{n}$ to be the pull-back of $\beta_{k}\to BF_{k}$ under the composition of projections $BF_{n}\to BF_{n-1}\to \dots\to BF_{k}$ of $\P^1$-bundles, where $k=0,\dots,n$.
\end{defn}

An equivalent definition of a bounded flag variety was given in \cite{bu-ra-98'} as follows. Choose a basis $e_{0},\dots,e_{n}$ in $\C^{n+1}$. Then $BF_{n}$ is the set of sequences $(l_{0},\dots,l_{n})$ of lines in $\C^{n+1}$ such that
\begin{equation}\label{eq:bfngeom}
l_{k}\subset l_{k-1}\oplus \C_{k},\ k=1,\dots,n,
\end{equation}
hold, where $\C_{k}:=\C\la e_{k}\ra$ denotes the line spanned by $e_{k}$ in $\C^{n+1}$. Put $l_{0}:=\C_{0}=\C\la e_0\ra$. The projection of the $\P^1$-bundle $BF_{n}\to BF_{n-1}$ from Definition \ref{defn:bfv} is given by $(l_{0},\dots,l_{n})\mapsto (l_{0},\dots,l_{n-1})$. Using \eqref{eq:bfngeom}, we obtain
\begin{equation}\label{eq:homcoord}
l_{k}\subset\C\la e_{0},\dots,e_{k}\ra,\ k=0,\dots,n,
\end{equation}
where $\C\la e_{0},\dots,e_{k}\ra$ denotes the linear span of vectors $e_{0},\dots,e_{k}$ in $\C^{n+1}$. Let $z_{k}:=[z_{k,0}:\dots:z_{k,k}]$ be the homogeneous coordinates of the line $l_{k}$ in \eqref{eq:homcoord}, where the coordinates $(z_{k,0},\dots,z_{k,k})$ are dual to $e_0,\dots,e_k$, for any $k=0,\dots,n$. In particular, $z_k=z_k(l_k)$, for any $k=0,\dots,n$. The embedding $BF_{n}\to \prod_{k=0}^{n}\P^k$ given by
\[
(l_{0},\dots,l_{n})\mapsto (z_0,z_{1},\dots,z_{n}),
\]
endows $BF_{n}$ with the tuple $(z_0,z_{1},\dots,z_{n})$ of homogeneous coordinates. The image of $BF_{n}$ in $\prod_{k=0}^{n}\P^k$ is given by the conditions
\begin{equation}\label{eq:bfneq}
\rk
\begin{pmatrix}
z_{k,0} & \dots & z_{k,k-1}\\
z_{k-1,0} & \dots & z_{k-1,k-1}
\end{pmatrix}
=1\mid k=2,\dots,n.
\end{equation}
These are quadratic equations (on the tuple of homogeneous coordinates $(z_0,z_{1},\dots,z_{n})$) given by vanishing of all $(2\times 2)$-minors of the matrices \eqref{eq:bfneq}.

It is well known that $BF_{n}$ is obtained from $\P^n$ by the sequence of blow-ups at strict transforms of the subvarieties $\lb z_{0}=\cdots=z_{k}=0\rb$ of $\P^n$ in any order, where $k$ runs over $\lbrace 1,\dots,n-1\rbrace$. The variety $BF_{n}$ is a nonsingular projective toric variety of dimension $n$ (see \cite{bu-ra-98'}, \cite{so-17}). The action of $\T^n=(\C^{\times})^n$ on $BF_{n}$, given by the formula
\begin{equation}\label{eq:bfnact}
(t_{1},\dots,t_{n})\circ z_{k}=[z_{k,0}:t_1 z_{k,1}:\dots:t_k z_{k,k}],\ k=1,\dots,n,\ (t_{1},\dots,t_{n})\in\T^n,
\end{equation}
has a dense open orbit.

The varieties $BR_{i,j}$ were introduced by V.M.~Buchstaber and N.~Ray in \cite{bu-ra-98} for any integers $i,j\geq 0$ such that $i\leq j$. They showed in \cite{bu-ra-98} that $BR_{i,j}$ is a nonsingular projective toric variety for any integers $i,j\geq 0$ such that $i\leq j$. We generalise their definition to the case of arbitrary integers $i, j\geq 0$, as follows.

\begin{defn}\label{def:brij}
For any integers $i, j\geq 0$ we call the hypersurface $BR_{i,j}$ in $BF_{i}\times \P^j$ given by the equation
\begin{equation}\label{eq:brijgen}
\sum_{k=0}^{\min\lb i,j\rb} z_{i,i-k}w_{j-k}=0,
\end{equation}
where $[w_{0}:\dots:w_{j}]$ are the homogeneous coordinates on the second factor $\P^j$ in $BF_{i}\times \P^j$, a \emph{generalised Buchstaber-Ray hypersurface}.
\end{defn}

\begin{rem}
Consider the hypersurface in $BF_{i}\times \P^j$ given by the equation
\begin{equation}\label{eq:singbr}
\sum_{k=0}^{\min\lb i,j\rb} z_{i,k}w_{k}=0.
\end{equation}
For any integers $i,j\geq 0$ such that $i\leq j$ the hypersurface given by \eqref{eq:singbr} is clearly isomorphic to  $BR_{i,j}$. However, unlike $BR_{2,1}$, the hypersurface given by \eqref{eq:singbr} is singular for $(i,j)=(2,1)$, see \cite{so-17}. Notice that $BR_{0,0}=\varnothing$, because substituting $0$ for $i,j$ in \eqref{eq:brijgen}, we obtain the equation $z_{0,0} w_{0}=0$ which has no solutions.
\end{rem}

Here is the definition of $BR_{i,j}$ in terms of configurations of lines in a complex vector space. Endow $\C^{\max\lb i,j\rb+1}$ with the natural Hermitian metric such that the standard basis $e_{0},\dots,e_{\max\lb i,j\rb}$ of $\C^{\max\lb i,j\rb+1}$ is orthonormal. Any point of $BF_{i}\times \P^{j}$ is the sequence $\bigl(l_0,\dots,l_i,l'\bigr)$ of lines in $\C^{\max\lb i,j\rb+1}$ satisfying the conditions
\begin{equation}\label{eq:geomdescbr}
l_{i-r}\subset l_{i-r-1}\oplus \C_{\max\lb i,j\rb-r},\ l'\subset \C\la e_{\max{\lb i,j\rb}-j},\dots,e_{j}\ra,
\end{equation}
for any integer $r=0,\dots,i-1$. Put $l_{0}:=\C_{\max\lb i,j\rb-i}$. Then $BR_{i,j}$ is given in $BF_{i}\times \P^{j}$ by the (algebraic) condition $l_{i}\perp \overline{l'}$, i.e. the lines  $l_{i}, \overline{l'}$ are orthogonal in $\C^{\max\lb i,j\rb+1}$.

\subsection{Ray hypersurface $R_{i,j}$}

We introduce the next definition by following \cite{ra-86}, \cite{so-17}.

\begin{defn}\label{def:sij}
For any integers $i, j\geq 0$, we call the hypersurface $R_{i,j}$ of $BF_{i}\times BF_{j}$ given by the equation
\begin{equation}\label{eq:rij}
\sum_{k=0}^{\min{\lb i,j\rb}}z_{i,i-k}w_{j,j-k}=0,
\end{equation}
where $(z_0,\dots,z_{i})$, $(w_{0},\dots,w_{j})$ are the tuples of homogeneous coordinates on $BF_{i}$, $BF_{j}$, respectively, a \emph{Ray hypersurface}.
\end{defn}

\begin{rem}\label{rem:sijdef}
The natural involution $BF_{i}\times BF_{j}\to BF_{j}\times BF_{i}$ maps $R_{i,j}$ to $R_{j,i}$. Hence, $R_{i,j}\simeq R_{j,i}$ for any integers $i,j\geq 0$. By definition, $R_{0,n+1}=BF_{n}$ and $R_{n,1}=BR_{n,1}$ for any integer $n\geq 0$. Notice that $R_{0,0}=\varnothing$, because substituting $0$ for $i,j$ in \eqref{eq:rij}, we get the equation $z_{0,0} w_{0,0}=0$ which has no solutions.
\end{rem}

Here is the definition of $R_{i,j}$ in terms of configurations of lines in a complex vector space. Any point in $BF_{i}\times BF_{j}$ is the sequence $\bigl(l_0,\dots,l_i,l'_0,\dots,l'_j\bigr)$ of lines in $\C^{\max\lb i,j\rb+1}$ satisfying the conditions
\begin{equation}\label{eq:geombrijcond}
l_{i-r}\subset l_{i-r-1}\oplus \C_{\max\lb i,j\rb-r}, \  l'_{j-q}\subset l'_{j-q-1}\oplus \C_{\max\lb i,j\rb-q},
\end{equation}
for any integers $r=0,\dots,i-1$ and $q=0,\dots,j-1$. Put $l_{0}:=\C_{\max\lb i,j\rb-i},\ l'_{0}:=\C_{\max\lb i,j\rb-j}$. Then $R_{i,j}\subset BF_{i}\times BF_{j}$ is given by the (algebraic) condition $l_{i}\perp \overline{l'_{j}}$.

\section{Monodromy in the weight graph of an algebraic torus action}\label{sec:torus_act}

\subsection{Definitions}\label{ssec:GKMdef}

Let us start this section by introducing the necessary notions.

\begin{defn}\label{def:diredge}[Compare with \cite{ba-14}]
Let $V$ be any finite set. Let $E_0$ be any finite collection of elements (a multiset, i.e. repetitions are allowed in $E_0$) of the set $2^{V(\Gamma)}$. Let $E:=\lb (f,v)|\ f\in E_0, v\in f\rb$. The pair $\Gamma=(V,E)$ is called an \emph{(abstract) hypergraph}. For any hypergraph $\Gamma=(V,E)$, any elements of $V(\Gamma):=V$, of $E_0$ and of $E(\Gamma):=E$ are called a \emph{vertex}, a \emph{hyperedge} and a \emph{pointed hyperedge}, respectively. Any element $f\in E_0$ such that $|f|=1$ is called a \emph{loop} of $\Gamma$. Any collection $f_1,\dots,f_k\in E_0$ is called a collection of \emph{multiple hyperedges} of $\Gamma$ if $f_1=\dots=f_k$. For any $e=(f,u)\in E(\Gamma)$, a vertex $i(e):=v$ is called an \emph{initial} vertex of a pointed hyperedge $e$. Put
\[
E_{v}(\Gamma)=\lb e\in E(\Gamma)|\ i(e)=v\rb.
\]
For any $e=(f,u)\in E(\Gamma)$ the elements $e$ and $f$ are called an \emph{oriented edge} and \emph{edge} of $\Gamma$, respectively, if $|f|=2$. If $e\in E(\Gamma)$ is an oriented edge, then the complementary vertex $t(e)$ of $e$ to $i(e)$ is called a \emph{terminal} vertex of $e$. In the following, we consider only those hypergraphs that have neither loops nor multiple hyperedges. Denote the oriented edge coming from $u$ to $v$ in $\Gamma$ by $E_{u}^{v}$ (if such an edge exists). In this case, put $\overline{e}=E_{v}^{u}$. If any hyperedge of $\Gamma$ is an edge, then $\Gamma$ is called a \emph{graph}.
\end{defn}

\begin{defn}
Let $\Gamma$ be any hypergraph. Denote by $G(\Gamma)$ the maximal \emph{subgraph} of the hypergraph $\Gamma$. Denote by $R(\Gamma)$ the subgraph of $\Gamma$ consisting of all edges in $\Gamma$ that have empty intersection with any hyperedge that is not an edge of $\Gamma$. We call $\Gamma$ an \emph{$n$-regular} hypergraph, if for any vertex $v$ of $R(\Gamma)$ one has $|E_{v}(G(\Gamma))|=n$.
\end{defn}

Clearly, $R(\Gamma)$ is a subgraph of $G(\Gamma)$. In general, this inclusion is strict.

\begin{ex}
Consider the edge graph of the tetrahedron with the set of vertices $V=\lb 1,2,3,4\rb$. Remove the edges corresponding to $\lb 1,2\rb,\lb 2,3\rb,\lb 3,1\rb$ and add the hyperedge $\lb 1,2,3\rb$ to this graph. Denote the obtained hypergraph by $\Gamma$. Clearly, $V(G(\Gamma))$ is $\lb 1,2,3,4\rb$, and the edges of $G(\Gamma)$ are $\lb 1,4\rb$, $\lb 2,4\rb$, $\lb 3,4\rb$. However, $V(R(\Gamma))=\lb 4\rb$, and $E(R(\Gamma))$ is empty.
\end{ex}

We introduce the notion of a weight hypergraph, motivated by notion of GKM-hypergraph (\cite{ba-14}) and GKM-graph (\cite{gu-za-01}), as follows. Let $\Gamma$ be any $n$-regular hypergraph. Let $\alpha\colon E(\Gamma)\to\Z^k$ be any map.

\begin{defn}[{cf. \cite{gu-za-01, ba-14}}]\label{def:conn}
We call $\alpha$ an \emph{axial function} on $\Gamma$, if the following conditions hold.

1) $\alpha(\overline{e})=-\alpha(e)$ for any edge $e\in E(G(\Gamma))$;

2) $\rk\Z\la \alpha(e)\colon e\in E_{v}(\Gamma)\ra=k$ for any $v\in V(\Gamma)$.\\
We call a pair $(\Gamma,\alpha)$ an \emph{$(n,k)$-type weight hypergraph} (or a \emph{weight hypergraph} for short, if the values of $k,n$ are clear from the context). We call the pair $(\Gamma,\alpha)$ a \emph{weight graph} if $\Gamma$ is a graph.
\end{defn}

Consider any collection $\nabla=\lb\nabla_{e}\colon e\in E(R(\Gamma))\rb$ of bijective maps $\nabla_{e}\colon E_{i(e)}(\Gamma)\to E_{t(e)}(\Gamma)$.

\begin{defn}[{cf. \cite{gu-za-01}}]
We call $\nabla$ a \emph{connection} on the weight hypergraph $(\Gamma,\alpha)$, if the following conditions hold for any $e\in E(R(\Gamma))$.

1) $\nabla_{\overline{e}}=(\nabla_{e})^{-1}$;

2) $\nabla_{e}(e)=\overline{e}$;

3) For any $e'\in E_{i(e)}(G(\Gamma))$ there exists an integer $c_{e}(e')\in\Z$ such that
\begin{equation}\label{eq:defconn}
\alpha(\nabla_{e} e')-\alpha(e')=c_{e}(e')\cdot\alpha(e).
\end{equation}
\end{defn}

\begin{rem}
A connection $\nabla$ on a weight hypergraph $(\Gamma,\alpha)$ consists of the maps $\nabla_{e}$, where $e$ exhausts the oriented edges of the graph $R(\Gamma)$. These maps act on the subsets of oriented edges of the graph $G(\Gamma)$.
\end{rem}

In order to study different connections on a given weight hypergraph, we give the following definition.

\begin{defn}\label{def:definite}
Let $(\Gamma,\alpha)$ be a weight hypergraph with a connection $\nabla$. For any edge $e$ of $E(R(\Gamma))$ we say that $(\Gamma,\alpha)$ is \emph{definite} at an edge $e$, if the affine lines $\alpha(e')+\R\la \alpha(e)\ra$ in the affine space $\mathbb{A}_{\R}^k$ are mutually different where $e'$ runs over $E_{i(e)}(\Gamma)\setminus \lb e\rb$. Otherwise, we call $(\Gamma,\alpha)$ \emph{nondefinite} at $e$. When $(\Gamma,\alpha)$ is clear from context, we call $e$ \emph{(non-)definite}, if $(\Gamma,\alpha)$ is (non-)definite at $e$, respectively. If $(\Gamma,\alpha)$ is definite at any edge of $R(\Gamma)$, then we call $(\Gamma,\alpha)$ a \emph{definite} weight hypergraph.
\end{defn}

The notion of definiteness of an edge $e$ is independent of an orientation of $e$ due to the following simple proposition.
\begin{pr}
Let $(\Gamma,\alpha)$ be a weight hypergraph with a connection $\nabla$. Let $e\in E(R(\Gamma))$ be an edge of \/ $\Gamma$. If \/ $(\Gamma,\alpha)$ is definite at $e$, then $(\Gamma,\alpha)$ is definite at $\overline{e}$, and the values of \/ $\nabla_e$ are uniquely determined by $(\Gamma,\alpha)$.
\end{pr}
\begin{proof}
Due to bijectivity of $\nabla_{e}$ and \eqref{eq:defconn}, one establishes the equality
\begin{equation}\label{eq:afflines}
\biggl\lb \alpha(e')+\R\la \alpha(e)\ra\colon e'\in E_{i(e)}(\Gamma),\ e'\neq e\biggr\rb=\biggl\lb \alpha(e'')+\R\la \alpha(\overline{e})\ra\colon e''\in E_{t(e)}(\Gamma),\ e''\neq \overline{e}\biggr\rb,
\end{equation}
of the sets of lines in the affine space $\A^k_{\R}$ by letting $e''=\nabla_{e} e'$, $e'\in E_{i(e)}(\Gamma),\ e'\neq e$. Hence, $(\Gamma,\alpha)$ is definite at $\overline{e}$. The set \eqref{eq:afflines} contains exactly $n-1$ elements because $\nabla$ is definite at $e$. One has $\nabla_{e}e'=e''$ iff the affine lines in $\A^k_{\R}$ corresponding to $e'\in E_{i(e)}(\Gamma)$ and $e''\in E_{t(e)}(\Gamma)$ by \eqref{eq:afflines} coincide. Hence, $\nabla_{e}$ is uniquely determined by $(\Gamma,\alpha)$.
\end{proof}

\begin{defn}(cf. \cite{gu-za-01}, \cite{ta-04})
A sequence $\gamma=( e_{1},\dots,e_{r})$ of edges in $G(\Gamma)$ is called an \emph{edge path}, if $t(e_j)=i(e_{j+1})$ for any $j=1,\dots,r-1$. For any edge path $\gamma=( e_{1},\dots,e_{r})$ in $G(\Gamma)$ the \emph{initial} and \emph{terminal} vertices of $\gamma$ are $i(\gamma):=i(e_1)$ and $t(\gamma):=t(e_r)$, respectively. Let $\gamma=( e_{1},\dots,e_{r})$ be any edge path in the \emph{subgraph} $R(\Gamma)$ of the hypergraph $\Gamma$. Then the \emph{parallel transport map} $\Pi_{\gamma}: E_{i(\gamma)}(\Gamma)\to E_{t(\gamma)}(\Gamma)$ of the connection $\nabla$ is defined by the formula $\Pi_{\gamma}(e):=\nabla_{e_r}\circ\dots\circ\nabla_{e_1} e$, where $e$ is any oriented edge from $E_{i(\gamma)}(\Gamma)$. If $i(\gamma)=t(\gamma)$, then $\Pi_{\gamma}$ is called the \emph{monodromy map} of $\nabla$ along $\gamma$.
\end{defn}

We generalise the notion of a face of a GKM-graph to the case of a nonregular subgraph in a weight hypergraph in the following two definitions.

\begin{defn}
Let $\Gamma'$ be a connected subgraph of $G(\Gamma)$. Let $e\in E(G(\Gamma))$ be any oriented edge satisfying $i(e)\in V(\Gamma')$. We call $e\in E(G(\Gamma))$ an \emph{internal} (\emph{external}, respectively) edge for $\Gamma'$ in $\Gamma$, if $t(e)\in V(\Gamma')$ ($t(e)\not\in V(\Gamma')$, respectively).
\end{defn}

In general, an internal edge $e\in G(\Gamma)$ for $\Gamma'$ may not belong to $E(\Gamma')$.

\begin{ex}\label{ex:p2gkm}
Consider the graph $\Gamma$ with the set of vertices $\lb 0,1,2\rb$, whose edges are $\lb 0,1\rb, \lb 1,2\rb, \lb 0,2\rb$. There exists a unique axial function $\alpha\colon E(\Gamma)\to\Z^2$ on $\Gamma$ such that $\alpha(E_0^1)=(0,-1)$, $\alpha(E_1^2)=(1,-1)$, $\alpha(E_2^0)=(0,1)$. Clearly, there exists a unique connection $\nabla$ on $(\Gamma,\alpha)$. Let $\Gamma'$ be the subgraph of $\Gamma$ with $V(\Gamma')=V(\Gamma)$, whose edges are $\lb 0,1\rb, \lb 1,2\rb$. Then the edge $E_2^0$ is internal for $\Gamma'$. However, $E_2^0\notin E(\Gamma')$.
\end{ex}

\begin{figure}
\centering
\caption{The internal edge $E_{0}^{2}$ to $\Gamma'$ does not belong to $\Gamma'$.}\label{fig:invsub}
\begin{tikzpicture}
\begin{scope}[every node/.style={circle,fill,inner sep=0pt, minimum size=6pt,node distance=6pt}]
    \node (1) at (-1,-1) [label={left:$0$ }]{};
    \node (2) at (0,1) [label={above:$1$ }]{};
    \node (3) at (1,-1) [label={right:$2$ }]{};
\end{scope}

\begin{scope}[
              every node/.style={fill=white,circle},
              every edge/.style={draw=black,very thick}]
    \path (1) edge (2);
    \path (2) edge (3);
\end{scope}

\begin{scope}[
              every node/.style={fill=white,circle},
              every edge/.style={draw=black,dashed,very thick}]
    \path (1) edge (3);
\end{scope}

\end{tikzpicture}
\end{figure}

\begin{defn}\label{def:inv}
Let $\Gamma$ be a connected $n$-regular hypergraph endowed with a connection $\nabla$. Let $\Gamma'$ be any connected subgraph of the graph $R(\Gamma)$. We call $\Gamma'$ an \emph{invariant} subgraph of $\Gamma$ with respect to $\nabla$, if the edge $\nabla_{e} e'\in E_{t(e)}(\Gamma)$ is internal for $\Gamma'$, where $e$ is any edge of $\Gamma'$ and $e'\in E_{i(e)}(\Gamma)$ is any internal edge for $\Gamma'$.
\end{defn}

Let us relate the above definitions with the notion from GKM-theory when $\Gamma$ is a graph.

\begin{defn}[\cite{gu-za-01}, \cite{bu-pa-15}]
The axial function $\alpha$ on $\Gamma$ is called \emph{$r$-independent}, if the vectors $\alpha(e_{1}),\dots,\alpha(e_{r})$ are linearly independent for any $v\in V(\Gamma)$ and any different $e_{1},\dots,e_{r}\in E_{v}(\Gamma)$. A weight graph $\Gamma$ endowed with an axial function $\alpha$ and a connection $\nabla$ is called a \emph{GKM-graph}, if $\alpha$ is $2$-independent. A connected $r$-regular \emph{subgraph} $\Gamma'$ of the GKM-graph $\Gamma$ is called an \emph{$r$-face} of $\Gamma$ (or a \emph{face}), if one has $\nabla_{e}(e')\in E(\Gamma')$ for any $v\in V(\Gamma')$ and any $e,e'\in E_{v}(\Gamma')$.
\end{defn}

It is well known that for any GKM-graph $(\Gamma,\alpha)$ with a $3$-independent axial function there exists no more than one connection $\nabla$ on it (e.g. see \cite{gu-za-01}).

\begin{rem}\label{rem:invgraphs}
Any face $\Gamma'$ of a GKM-graph $(\Gamma,\alpha)$ with a connection $\nabla$ is invariant under $\nabla$ in sense of Definition \ref{def:inv}. (We distinguish between the notion of a face of a GKM-graph \cite{gu-za-01} and its generalisation from Definition \ref{def:inv}, namely, the notion of an invariant subgraph in a weight hypergraph.) Let $(\Gamma,\alpha)$ be any weight hypergraph. Let $\Gamma'$ be any connected subgraph of $R(\Gamma)$. It is easy to prove that $\Gamma'$ is invariant under $\nabla$ iff for any edge $e$ of $\Gamma'$ and any external edge $e'\in E_{i(e)}(\Gamma)$ for $\Gamma'$ the edge $\nabla_{e} e'\in E_{t(e)}(\Gamma)$ is external for $\Gamma'$. For any edge path $\gamma$ in any invariant subgraph $\Gamma'$ of $R(\Gamma)$ if an edge $e\in E_{i(\gamma)}(\Gamma)$ is internal (external, respectively) for $\Gamma'$, then $\Pi_{\gamma}(e)$ is internal (external, respectively) for $\Gamma'$. Let us finally remark that, in general, an invariant subgraph is not regular. Following the notation of Example \ref{ex:p2gkm}, the nonregular subgraph $\Gamma'$ of $\Gamma$ is invariant for $\nabla$, because the set of external edges to $\Gamma'$ in $\Gamma$ is empty, see Fig. \ref{fig:invsub}.
\end{rem}

\subsection{Weight hypergraph of a complex $(\C^{\times})^k$-manifold}\label{ssec:geomcon}

Let $\T^{k}\simeq(\C^{\times})^{k}$ be the algebraic (i.e. noncompact) torus acting effectively by biholomorphic maps on a compact connected complex manifold $X^{n}$, where $n,k\geq 0$. Denote by $X^{\T^k}$ the set of fixed points of this action.

\begin{asm}\label{asm:torusact}
The manifold $X^n$ has an open cover by its open complex $\T^k$-invariant submanifolds $U(x)$, where $x\in X^{\T^k}$. One has $U(x)^{\T^k}=\lb x\rb$ for any $x\in X^{\T^k}$. For any $x\in X^{\T^k}$ there exists a $\T^k$-equivariant biholomorphism $\phi_{x}\colon U(x)\to \C^n$. The action of $\T^k$ on $\C^n$ here is induced by a monomorphism $\iota\colon \T^k\to\T^n$ such that $\T^n$ is a direct product of $\iota(\T^k)$ and some algebraic torus. The $\T^n$-action on $\C^n$ here is given by the formula
\begin{equation}\label{eq:stand}
(t_{1},\dots,t_{n})\circ(z_1,\dots,z_n)=(t_1 z_1,\dots,t_n z_n),\ (t_{1},\dots,t_{n})\in\T^n,\ (z_1,\dots,z_n)\in \C^n.
\end{equation}
\end{asm}

\begin{rem}
Assumption \ref{asm:torusact} implies that the set of fixed points $X^{\T^k}$ is finite and nonempty, and that the $\T^k$-stabiliser of any point $x\in X$ is a direct factor of $\T^k$, that is an algebraic subtorus.
\end{rem}

The induced representation of $\T^k$ on the tangent space $T_{x} X^{n}$ at any fixed point $x\in X^{\T^k}$ decomposes into the sum
\begin{equation}\label{eq:tangmfd}
T_{x}X^{n}=\bigoplus_{j=1}^{n}V(w_{j}),
\end{equation}
of characters corresponding to the primitive nonzero elements $w_{1},\dots, w_{n}\in\Hom(\T^k,\T^1)\simeq \Z^k$. These vectors are called the \emph{weights} of the $\T^k$-action on $X$ at the fixed point $x\in X^{\T^k}$.

For any $x\in X^{\T^k}$ and any $l\in\P(\Z^k)$ let $Y=Y(x,l)\subseteq X$ be the connected component of $X^{\Ker l}$ such that $x\in Y$ (notice that there exists a unique $Y$ for any $x,l$). The $\T^k$-action on $X$ induces the effective action of the algebraic torus $\T^k/\Ker l\simeq \C^{\times}$ on $Y$.

\begin{rem}
For any $l\in \P(\Z^k)$ such that $l$ is not represented by a weight of the $\T^k$-action at $x$, the set $X^{\ker l}=X^{\T^k}$ is finite and zero-dimensional.
\end{rem}

For any $x\in X^{\T^k}$ let $w_{j_1},\dots, w_{j_{q}}$ be all weights of the $\T^k$-action at $x$ that are $(\pm 1)$-multiples of $w$ for some $q=q(x,w)\in\Z$, that is, $w_{j_i}=\pm w$ for all $i=1,\dots,q$. For any nonzero element of $w\in\Z^k$ denote the corresponding class in $\P\Z^k$ by $[w]$.

\begin{pr}\label{pr:2skel}
Suppose that Assumption \ref{asm:torusact} holds for the $\T^k$-action on $X$. Then for any $x\in X^{\T^k}$ and any nonzero $w\in \Z^k$ the set $Y=Y(x,[w])$ has a structure of a complex $\T^k$-invariant closed submanifold of \/ $X$. One has $q=q(x,w)=\dim Y$ and
\begin{equation}\label{eq:tangspace}
T_{x} Y=\bigoplus_{r=1}^{q}V(w_{j_{r}})\subseteq T_{x}X^{n}.
\end{equation}
\end{pr}
\begin{proof}
For any $y\in X^{\T^k}$ the linear subspace $(\C^n)^{\Ker w}$ of $\C^n$ coincides with the linear subspace $\phi_{y}(Y(x,[w])\cap U(y))$ (see Assumption \ref{asm:torusact}). This implies all statements of the proposition.
\end{proof}

The following fact is well known.

\begin{pr}
Any $1$-dimensional $\T^k$-invariant complex submanifold of \/ $X$ is equivariantly biholomorphic to the standard \/ $\C^{\times}$-action on \/ $\P^1$ having weights $k,-k$ for some nonzero $k\in \Z$.
\end{pr}

We assign a weight hypergraph to any effective $\T^k$-action on any compact connected complex manifold $X^{n}$ satisfying Assumption \ref{asm:torusact}, as follows. (Compare with \cite{gu-za-01g}, \cite{ba-14}.)

\begin{constr}[Weight hypergraph of an algebraic torus action, compare with \cite{ba-14}]
Let $W\subset \P(\Z^k)$ be the (finite by compactness of $X$) set of all elements represented by a weight at some $\T^k$-fixed point of the $\T^k$-action on $X$.
Put
\[
V:=X^{\T^{k}},\ E:=\big\lb Y(x,l)^{\T^k}\big|\ x\in X^{\T^k},\ l\in W\big\rb.
\]
Here we regard $E$ as a finite multiset (due to compactness of $X$). Notice that $\Gamma:=(V,E)$ is a connected hypergraph. Denote the submanifold $Y=Y(x,l)$ of $X$ corresponding to a hyperedge $e\in E(\Gamma)$ by $Y(e)$ for any $e\in E(\Gamma)$. For any $e\in E(\Gamma)$ let $\alpha(e)$ be any weight of the corresponding $\T$-action on $Y(e)$ at the fixed point $i(e)$ (in general, $\alpha(e)$ is defined up to sign). Notice that $\alpha$ is an axial function on $\Gamma$. We call $(\Gamma, \alpha)$ the ($(n,k)$-type) weight hypergraph $(\Gamma, \alpha)$ \emph{associated with the action} of $\T^k$ on $X^n$.
\end{constr}

In the following, we consider only the class of $\T^k$-actions such that the associated hypergraphs have neither loops, nor multiple hyperedges. This implies that for any associated hypergraph $(\Gamma,E)$ the multiset $E$ is a set.

\begin{rem}
Let $e\in E(\Gamma)$ be a hyperedge of the associated weight hypergraph $(\Gamma,\alpha)$ of the $\T^k$-action on $X$. If $e$ is an edge (that is, $\dim Y(e)=1$) of $\Gamma$, then $\alpha(e)$ is uniquely defined by the $\T^k$-action on $X$. In general, $\alpha(e)$ is defined for the $\T^k$-action on $X$ only up to a sign.
\end{rem}

We define the connection on the weight hypergraph $(\Gamma,\alpha)$ associated with the $\T^k$-action on $X$ by following the construction from \cite{gu-za-01}, as follows.

\begin{constr}[Connection on a weight hypergraph of an algebraic torus action]\label{constr:conn}
Let $e\in E(R(\Gamma))$ be any edge. Consider any $\T^k$-invariant rational curve $Y$ of $X$ with different fixed points $x,y\in Y$. Let $E_{x}(\Gamma)=\lb e'_1,\dots,e'_n\rb$ and $E_{y}(\Gamma)=\lb e''_1,\dots,e''_n\rb$. Let $\alpha(e'_{j})=w'_j$, $\alpha(e''_{j})=w''_j\in\Z^k$ be the weights of the $\T^k$-action on $X$ at fixed points $x,y$, respectively, where $j=1,\dots,n$. Any complex vector bundle over $Y$ splits equivariantly into the direct sum
\[
(T X^{n})|_{Y}=\bigoplus_{j=1}^n \xi_{j},
\]
of $\T^k$-equivariant complex line bundles $\xi_{j}$ over $Y$. Hence, there exist permutations $\sigma,\tau$ of $\lbrace 1,2,\dots,n\rbrace$ such that $(\xi_{j})_{x}=V(w'_{\sigma(j)})$, $(\xi_{j})_{y}=V(w''_{\tau(j)})$. We put $\nabla_{e} e'_{j}:=e''_{\tau^{-1}\circ\sigma (j)}$ for any $j=1,\dots,n$. One can check that the collection $\nabla_{e}$, $e\in E(R(\Gamma))$ is a connection on the weight hypergraph $(\Gamma,\alpha)$.
\end{constr}

\begin{rem}
In general, a connection on a weight graph, associated with a torus action on a complex manifold, is not unique, because there is freedom in choosing the permutations $\sigma, \tau$ from Construction \ref{constr:conn}, see Example \ref{ex:depbr} below. However, if an associated weight graph is definite, then it uniquely determines a connection on it.
\end{rem}

\subsection{GKM-graph of a nonsingular projective toric variety}\label{ssec:toricconn}

Let $X^{n}$ be a nonsingular projective toric variety of dimension $n\geq 3$. The weight graph $(\Gamma, \alpha)$ and the connection $\nabla$ associated with the natural $\T^n$-action on $X^{n}$ coincide with the associated GKM-graph (with the natural connection) which is given as follows \cite{gu-za-01}. The graph $\Gamma$ is the edge graph of the simple moment polytope $P^n\subset\R^n$ of $X^{n}$, where $\R^n=\Z^n\otimes_{\Z}\R$ (see \cite{bu-pa-15}). For any edge $e$ of $\Gamma$ the vector $\alpha(e)\in\Z^n\subset\R^n$ is emanating from $i(e)$ to $t(e)$ being parallel to the corresponding edge of the polytope $P^n$. The axial function $\alpha$ is $n$-independent, because $P^n$ is a simple polytope. Hence the weight graph $(\Gamma,\alpha)$ admits a unique connection.

The faces of the graph $\Gamma$ with the connection $\nabla$ are described by the following lemma.

\begin{lm}\cite[Lemma 7.9.7, p.306]{bu-pa-15}\label{lm:spanface}
For any $v\in V(\Gamma)$, any integer $k\geq 0$ and any distinct elements $e_{1},\dots,e_k\in E_{v}(\Gamma)$ there exists a unique $k$-face $G$ of \/ $\Gamma$ containing $e_{1},\dots,e_k$. In particular, $G$ is the edge graph of a polytopal face of the moment polytope of \/ $X^n$.
\end{lm}

It is straight-forward to deduce the following lemma from convexity of faces for the moment polytope $P$.
\begin{lm}\label{lm:nochord}
Let $G\subseteq P$ be a face of the moment polytope $P$ of \/ $X^n$. If $u,v\in V(G)$ are connected by an edge $e$ of the polytope $P$, then $e\subseteq G$. In particular, for any two faces $F_1,F_2$ of the edge graph $\Gamma$ of \/ $P^n$ if \/ $V(F_1)=V(F_2)$ then $F_1=F_2$.
\end{lm}

\begin{pr}\label{pr:monodromy}
Let $\Gamma'$ be any face of the GKM-graph $\Gamma$ of \/ $X^n$. Let $\gamma$ be any edge path in $\Gamma'$. Then one has
\begin{equation}\label{eq:extedges}
\Pi_{\gamma}\bigl(E_{i(\gamma)}(\Gamma)\setminus E_{i(\gamma)}(\Gamma')\bigr)=E_{t(\gamma)}(\Gamma)\setminus E_{t(\gamma)}(\Gamma').
\end{equation}
If $i(\gamma)=t(\gamma)$, then the well-defined (by \eqref{eq:extedges}) restriction of the monodromy map $\Pi_{\gamma}$ to $E_{i(\gamma)}(\Gamma)\setminus E_{i(\gamma)}(\Gamma')$ is the identity map.
\end{pr}
\begin{proof}
By Lemma \ref{lm:spanface}, for any $e\in E(\Gamma)$ there exists a unique $(n-1)$-face $\Gamma(e)$ of $\Gamma$ such that $i(e)\in V(\Gamma(e))$ and $e\not\in E(\Gamma(e))$.
Let $e\in E_{i(\gamma)}(\Gamma)\setminus E_{i(\gamma)}(\Gamma')$. Then there exists a unique edge $e'\in E_{t(\gamma)}(\Gamma)\setminus E_{t(\gamma)}(\Gamma')$ such that $\Gamma(e)=\Gamma(e')$. We conclude that $\Pi_{\gamma}(e)=e'$, because $\Gamma(e)$ is invariant. In particular, if $i(\gamma)=t(\gamma)$, then $e'=e$. This completes the proof of the proposition.
\end{proof}

Let $\iota\colon \T^{k}\to \T^{n}$ be any monomorphism of tori. Suppose that Assumption \ref{asm:torusact} holds for the induced $\T^k$-action on the toric variety $X^{n}$. Then the weight hypergraph $(\Gamma',\alpha')$ associated with this $\T^{k}$-action on $X$ is well-defined.

\begin{rem}
Any $\T^{n}$-invariant submanifold of $X^n$ is $\T^{k}$-invariant. The opposite is false. For example, the Milnor hypersurface $H_{i,j}$ is invariant under the restriction of the action of the respective algebraic subtorus $\T^{\max{\lb i,j\rb}}$ in $\T^i\times\T^j$. However, for any integers $i,j\geq 1$ the hypersurface $H_{i,j}$ is not invariant under the natural $(\T^i\times\T^j)$-action on $\P^i \times\P^j$, see \eqref{eq:hij_act}.
\end{rem}

\begin{pr}\label{pr:invsubm}
Let $k\geq 2$. Then one has $X^{\T^{n}}= X^{\T^{k}}$, and any $\T^n$-invariant rational irreducible curve of \/ $X$ is $\T^k$-invariant. In particular, one has $E_{v}(\Gamma')=E_{v}(\Gamma)$ for any vertex $v$ of \/ $R(\Gamma')$.
\end{pr}
\begin{proof}
The inclusion $X^{\T^{n}}\subseteq X^{\T^{k}}$ holds, because any $\T^{n}$-invariant submanifold of $X^{n}$ is $\T^{k}$-invariant. To prove the first claim, it remains to note that the integers $|X^{\T^{n}}|,|X^{\T^{k}}|$ are equal to the Euler characteristic of $X^{n}$ (see \cite{gu-10}). Let $p\colon \Z^{n}\to\Z^k$ be the homomorphism of character lattices corresponding to the monomorphism $\iota$ of tori. Let $v=x\in V(R(\Gamma'))$. Any $\T^n$-invariant irreducible rational curve of $X$ has the form $Y(x,[w])$ for some weight $w\in\Z^n$ at $x\in X^{\T^n}$. Let $Y(x,[w])$ be such a curve. Clearly, $Y(x,[w])$ is $\T^k$-invariant. Hence, $Y(x,[w])\subseteq Y(x,[p(w)])$, where $Y(x,[p(w)])$ is the $\T^k$-invariant submanifold of $X$. The submanifold $Y(x,[p(w)])$ is a rational irreducible curve, because $v\in V(R(\Gamma'))$. Hence, $Y(x,[w])=Y(x,[p(w)])$. This proves the second claim of the proposition.
\end{proof}

\section{Algebraic torus actions on $BR_{i,j}$, $R_{i,j}$, and proofs of Theorems \ref{thm:brijnottoric}, \ref{thm:rijnontoric}}\label{sec:appl}

Throughout this section we refer to some auxiliary results from Appendix \ref{sec:blowup}.

\subsection{Generalised Buchstaber-Ray hypersurface $BR_{i,j}$}\label{ssec:genbrij}

Let us start by recalling the description of $\T^n$-fixed points in the bounded flag manifold $BF_n$. For any $k=0,\dots,n$ and any $\underline{u}=(u_1,\dots,u_n)\in\F_{2}^n$ put
\[
a_{k}(\underline{u}):=\max\biggl(\lbrace 0\rbrace\cup {\big\lb r\in\lb 1,\dots,k\rb\colon u_{r}=1\big\rb}\biggr).
\]
For any $k=1,\dots,n$ let $b_{k}(\underline{u})$ be a unique integer such that $\lb a_{k}(\underline{u}), b_{k}(\underline{u})\rb=\lb a_{k-1}(\underline{u}), k\rb$ holds. Let
\[
\C_{\underline{u}}:=(\C_{a_{1}(\underline{u})},\dots, \C_{a_{n}(\underline{u})})\in BF_{n},
\]
where $\C_j$ is the line spanned by $j$-th vector of the standard basis in $\C^{n+1}$, $j=0,\dots,n$ (see \S \ref{ssec:brij}). The following two lemmas are straight-forward to prove.
\begin{lm}\label{lm:tuples}
For any $\underline{u}\in\F_{2}^n$ and any integer $k=0,\dots,n$ one has the identity
\[
\lb b_{1}(\underline{u}),\dots, b_{k}(\underline{u}),a_{k}(\underline{u})\rb=\lb 0,1,\dots,k\rb.
\]
\end{lm}

\begin{lm}[\cite{bu-ra-98},\cite{bu-pa-15}]
One has $(BF_{n})^{\T^n}=\lb \C_{\underline{u}}|\ \underline{u}\in\F_{2}^n\rb$.
\end{lm}

For any $\underline{u}\in\F_2^n$ let $U_{\underline{u}}:=\lb l_{k}\neq\C_{a_{k}(\underline{u})}|\ k=1,\dots,n\rb$. Clearly, $U_{\underline{u}}=\lb z_{k,a_k(\underline{u})}\neq 0|\  k=1,\dots,n\rb$  is an affine subvariety of $BF_{n}$, where $(z_0,\dots,z_n)$ is the tuple of homogeneous coordinates on $BF_n$ (see \S \ref{ssec:brij}). Hence, $U_{\underline{u}}$ is $\T^n$-invariant with respect to the action \eqref{eq:bfnact} for any $\underline{u}\in\F_2^n$. It is easy to deduce the following lemma by the induction on $n\geq 0$ from the equations \eqref{eq:bfneq}.

\begin{lm}\label{lm:bfncoordcharts}
For any $\underline{u}\in\F_{2}^n$ the invariant affine subvariety $U_{\underline{u}}$ of the toric variety $BF_n$ is equivariantly isomorphic to \/ $\C^n$ with the $\T^n$-action \eqref{eq:stand} under the following isomorphism
\[
U_{\underline{u}}\to\C^n,\ (z_0,\dots,z_n)\mapsto \biggl(\frac{z_{1,b_1(\underline{u})}}{z_{1,a_1(\underline{u})}},\dots,\frac{z_{n,b_n(\underline{u})}}{z_{n,a_n(\underline{u})}}\biggr).
\]
\end{lm}

Recall that the projective space $\P^n$ is covered by its open subvarieties $U_k:=\lb w_{k}\neq 0\rb$, $k=0,\dots,n$, where $[w_0:\cdots:w_n]\in \P^n$. These subvarieties are invariant under the standard $\T^n$-action \eqref{eq:pn_act} on $\P^n$.  Any $(\C^{\times})^n$-invariant irreducible rational curve of $\P^n$ has the form $\P^1(k,q)=\lb \C\la \lambda e_{k}+\mu e_{q}\ra\in \P^{n}|\ [\lambda:\mu]\in \P^1\rb$, where $k,q=0,\dots,n$ are any integers such that $k\neq q$. For any vectors $\underline{u},\underline{v}\in\F_{2}^{n}$ and any $[\lambda:\mu]\in\P^1$ let
\[
\lambda\C_{\underline{u}}+\mu\C_{\underline{v}}:=\big(\C\la \lambda e_{a_{1}(\underline{u})}+\mu e_{a_{1}(\underline{v})}\ra,\dots, \C\la \lambda e_{a_{n}(\underline{u})}+\mu e_{a_{n}(\underline{v})}\ra\big)\in BF_n.
\]
Under the action \eqref{eq:bfnact} any $(\C^{\times})^n$-invariant irreducible rational curve of $BF_n$ has the form
\[
\P^1(\underline{u},q):=\big\lb \lambda\C_{\underline{u}}+\mu\C_{\underline{u}+1_{q}}|\ [\lambda:\mu]\in \P^1\big\rb,
\]
where $q=1,\dots,n$ and $\underline{u}\in\F_{2}^{n}$ are arbitrary. Here $1_{q}\in \F_{2}^{n}$ has all zero coordinates besides $q$-th coordinate that is equal to $1$. The following proposition is easily deduced from Lemma \ref{lm:bfncoordcharts}.

\begin{pr}
For any $\underline{u}\in\F_{2}^n$ the weights of the $(\C^{\times})^{n}$-action \eqref{eq:bfnact} on $BF_{n}$ at the fixed point \/ $\C_{\underline{u}}$ are $e_{b_{q}(\underline{u})}-e_{a_{q}(\underline{u})}$, where $q$ runs over $\lbrace 1,\dots,n\rbrace$.
\end{pr}

For any integers $i,j\geq 0$ such that $i\geq j$, the formula
\begin{equation}\label{eq:brij_act}
(t_{1},\dots,t_{i})\circ(z_i,w)=([z_{i,0}:t_1 z_{i,1}:\dots:t_i z_{i,i}],[t_{i-j}^{-1} w_0:t_{i-j+1}^{-1}w_1:\dots:t_i^{-1} w_j]),\ (t_{1},\dots,t_{i})\in\T^i,
\end{equation}
determines a unique effective action of the algebraic torus $(\C^{\times})^{i}$ on hypersurface $BR_{i,j}$ in the coordinates of $BF_{i}\times \P^j$. This follows easily from \eqref{eq:brijgen} together with the relations \eqref{eq:bfneq} on the tuple of homogeneous coordinates $(z_0,\dots,z_i)$ on $BF_{i}$. The hypersurface $BR_{i,j}$ is an invariant subvariety of $BF_i\times \P^{j}$ with respect to the action \eqref{eq:brij_act} of the algebraic subtorus $\T^i$ in $\T^i\times\T^j$. Hence, the fixed point set of the $\T^i$-action \eqref{eq:brij_act} on $BR_{i,j}$ is the subset of fixed points of the toric variety $BF_i\times \P^{j}$. It can easily be checked that $BR_{i,j}^{\T^i}$ consists of the points $x_{\underline{u},k}:=(\C_{\underline{u}},\C_{k})\in BR_{i,j}$ for any $\underline{u}\in \F_{2}^i$ and any $k=0,\dots,j$ such that $a_i(\underline{u})\neq k+(i-j)$ holds. It follows that the open covering of $BR_{i,j}$ by the open $\T^{i}$-invariant subvarieties $(U_{\underline{u}}\times U_k)\cap BR_{i,j}$, where $\underline{u} \in\F_{2}^i$, $k=0,\dots,j$ are any elements such that $a_i(\underline{u})\neq k+i-j$, satisfies the Assumption \ref{asm:torusact}.

Denote the combinatorial equivalence class of the standard simplex $\lb (x_1,\dots,x_n)\in \R^n\colon \sum_{j=1}^n x_j=1\mid x_j\geq 0,\ j=1,\dots,n\rb$ in $\R^n$ by $\Delta^n$. Let $I^n=(\Delta^1)^n$ be the Cartesian product of $n$ copies of $\Delta^1$.

\begin{pr}\label{pr:toricbr}
$(i)$ For any integers $i,j\geq 0$ such that $i\leq j$ the variety $BR_{i,j}$ is a projective toric variety which is an algebraic $\P^{j-1}$-bundle over \/ $BF_{i}$. Its moment polytope is combinatorially equivalent to \/ $I^i\times\Delta^{j-1}$;

$(ii)$ For any integer $n\geq 0$, the variety $BR_{n+1,0}\simeq BF_{n}$ is a projective toric variety whose moment polytope is combinatorially equivalent to $I^n$. In particular, $BR_{n+1,0}$ is a Bott tower;

$(iii)$ For any integer $n\geq 2$, the variety $BR_{n,1}$ is a projective toric variety whose moment polytope is combinatorially equivalent to the truncation of \/ $I^n$ at its face $I^{n-2}$ (see \cite{bu-pa-15}).
\end{pr}
\begin{proof}
For the proof of $(i)$ see \cite{bu-ra-98} or \cite[p.350]{bu-pa-15}. The claim $(ii)$ follows from the Definition \ref{def:brij}. By Theorem \ref{thm:1blowbr}, the variety $BR_{n,1}$ is the blow-up of $BF_{n-1}\times \P^1$ along the zero locus $\lb z_{n-1,n-1}=w_{1}=0\rb$, which is invariant under the action \eqref{eq:brij_act} and is isomorphic to $BF_{n-2}$. Hence, the blow-up $BR_{n,1}\to BF_{n-1}\times \P^1$ is $\T^n$-equivariant. In particular, $BR_{n,1}$ is a projective toric variety and the respective moment polytope is obtained by the truncation indicated above.
\end{proof}

Notice that the fan of any projective nonsingular toric variety is the normal fan of the respective moment polytope.

\begin{rem}\label{rem:br21}
By Proposition \ref{pr:brijdef} $(iii)$, the blow-up $BR_{2,1}\to \widehat{H}_{2,1}$ is $\T^2$-equivariant, where $\widehat{H}_{2,1}=\P(\mc{O}(-1)\oplus\underline{\C})\to \P^1$ is a toric surface. By Theorem \ref{thm:1blowbr}, the blow-up $BR_{2,1}\to \P^1\times\P^1$ is also $\T^2$-equivariant. The two $\T^2$-actions on $BR_{2,1}$ obtained in this way coincide. Let $\Sigma$ be the fan in $\R^2$ corresponding to the toric variety $BR_{2,1}$. It is easy to show that the generators of the one-dimensional cones from $\Sigma$ are the columns of the following matrix
\[
\begin{pmatrix}
1 & -1 & 0 & 0 & -1\\
0 & 0 & 1 & -1 & 1\\
\end{pmatrix}.
\]
\end{rem}

For any integer $q=1,\dots,i$ and any $\underline{u}\in\F_{2}^i$ denote by $b(q)=b(\underline{u},q)$ the vector $e_{b_{q}(\underline{u})}-e_{a_{q}(\underline{u})}\in\Z^i$. For any integers $k,r=0,\dots,j$ and any $\underline{u}\in\F_{2}^j$ denote by $b'(r)=b'(k,r)$ the vector $e_{k+i-j}-e_{r+i-j}\in\Z^i$. It is easy to prove the following two propositions.

\begin{pr}\label{pr:weightbr}
Let $i,j\geq 0$ be any integers such that $i\geq j$. Then for any $\underline{u}\in\F_{2}^i$ and any $k=0,\dots,j$ such that $a_i(\underline{u})\neq k+(i-j)$ the weights of the $\T^{i}$-action \eqref{eq:brij_act} on $BR_{i,j}$ at the fixed point $x_{\underline{u},k}$ are the elements of the multiset
\begin{equation}\label{eq:brijweights}
\big\lb b(\underline{u},q)\big|\ q=1,\dots,i\big\rb\cup\big\lb b'(k,r)\big|\ r=0,\dots,j,\ r\neq k\big\rb\setminus\big\lb e_{k+(i-j)}-e_{a_{i}(\underline{u})}\big\rb.
\end{equation}
\end{pr}

\begin{rem}
If $a_{i}(\underline{u})<k+(i-j)$, then $b(k+(i-j))= e_{k+(i-j)}-e_{a_{i}(\underline{u})}$. If $a_{i}(\underline{u})>k+(i-j)$, then $b'(r)= e_{k+(i-j)}-e_{a_{i}(\underline{u})}$, where $r=a_{i}(\underline{u})-(i-j)$. This justifies the exclusion in \eqref{eq:brijweights}.
\end{rem}

\begin{pr}\label{pr:weightsbrij}
Let $\underline{u}\in\F_{2}^i$, $k=0,\dots,j$ be any elements such that $a_i(\underline{u})\neq k+(i-j)$. Then the multiset of collections of pairwise proportional weights of the $\T^{i}$-action \eqref{eq:brij_act} on $BR_{i,j}$ at $x_{\underline{u},k}$ consists of the multiset of the (unordered) pairs $b(q), b'(r)$ of weights, where $q=1,\dots,i$ and $r=0,\dots,j$ are any integers satisfying the following conditions
\[
\lb a_{q}(\underline{u}), b_{q}(\underline{u})\rb=\lb k+(i-j), r+(i-j)\rb\neq \lb k+(i-j),a_{i}(\underline{u})\rb.
\]
The $\T^i$-invariant subvariety $Y=Y(x_{\underline{u},k},[b(q)])$ of \/ $BR_{i,j}$ corresponding to the weight $b(q)\in\Z^i$ (see Section \ref{sec:torus_act}) is \/ $\P^1(\underline{u},q)\times \P^1(k+(i-j),r+(i-j))\subseteq BF_{i}\times \P^{j}$. One has $Y^{\T^i}=\lb x_{\underline{u},k},\ x_{\underline{u}+1_{q},k},\ x_{\underline{u},r},\ x_{\underline{u}+1_{q},r}\rb$.
\end{pr}

The following example shows that the $4$-dimensional variety $BR_{3,2}$ has a fixed point of the $\T^3$-action \eqref{eq:brij_act} whose weights are linearly dependent.

\begin{ex}\label{ex:depbr}
The weights of the $\T^3$-action \eqref{eq:brij_act} on $BR_{3,2}$ at the fixed points $x_{111,0}$, $x_{111,1}$, $x_{101,0}$, $x_{101,1}$ are the respective collections of vectors in $\Z^3$ given as follows.
\begin{itemize}
\item $(1,-1,0)$, $(-1,0,0)$, $(1,-1,0)$, $(0,1,-1)$;
\item $(-1,0,0)$, $(1,-1,0)$, $(0,1,-1)$, $(-1,1,0)$;
\item $(1,-1,0)$, $(-1,0,0)$, $(-1,1,0)$, $(1,0,-1)$;
\item $(-1,0,0)$, $(-1,1,0)$, $(1,0,-1)$, $(-1,1,0)$.
\end{itemize}
\end{ex}

For any integers $i,j\geq 0$ such that $i>j$, let $(\Gamma,\alpha)=(\Gamma(BR_{i,j}),\alpha(BR_{i,j}))$ be the weight hypergraph associated with the $\T^i$-action \eqref{eq:brij_act} on $BR_{i,j}$ (notice that the Assumption \ref{asm:torusact} is satisfied for such an action).

\begin{pr}\label{pr:2sphbrij}
Let $i,j\geq 0$ be any integers such that $i>j$. Let $k=0,\dots,j$ and $\underline{u}\in\F_2^{i}$ be arbitrary. Then:

$(i)$ For any integer $p=0,\dots, j$ satisfying $p\neq a_i(\underline{u})-(i-j),k$, the hypergraph $\Gamma$ has a pointed hyperedge $E$ such that $x_{\underline{u},k},x_{\underline{u},p}\in E$ and $\alpha(E)=\pm(e_{k+(i-j)}-e_{p+(i-j)})$;

$(ii)$ For any integer $q=1,\dots, i$ satisfying $a_{i}(\underline{u}+1_{q})\neq k+(i-j)$, the hypergraph $\Gamma$ has a pointed hyperedge $E$ such that $x_{\underline{u},k},x_{\underline{u}+1_{q},k}\in E$ and $\alpha(E)=\pm(e_{b_{q}(\underline{u})}-e_{a_{q}(\underline{u})})$;

$(iii)$ If there exist integers $r=1,\dots,i$ and $s=0,\dots,j$ satisfying $a_{i}(\underline{u}+1_{r})=k+(i-j)$ and $a_{i}(\underline{u})=s+(i-j)$, then the hypergraph $\Gamma$ has a  pointed hyperedge $E$ such that $x_{\underline{u},k},x_{\underline{u}+1_{r},s}\in E$ and $\alpha(E)=\pm(e_{k+(i-j)}-e_{a_{i}(\underline{u})})$.
\end{pr}
\begin{proof}
It is not hard to prove that any of the following irreducible rational curves
\[
\biggl\lb\big(\C_{\underline{u}},\C\la\lambda\cdot e_{k}+\mu\cdot e_p\ra\big)\biggl|\ [\lambda:\mu]\in \P^1\biggr\rb,\ w=\pm(e_{k+(i-j)}-e_{p+(i-j)}),
\]
\[
\biggl\lb\big(\lambda\cdot \C_{\underline{u}}+\mu\cdot \C_{\underline{u}+1_{q}},\C\la e_k\ra\big)\biggl|\ [\lambda:\mu]\in \P^1\biggr\rb,\ w=\pm(e_{b_{q}(\underline{u})}-e_{a_{q}(\underline{u})}),
\]
\[
\biggl\lb\big(\lambda\cdot \C_{\underline{u}}+\mu\cdot \C_{\underline{u}+1_{r}},\C\la\mu\cdot e_{s+(i-j)}-\lambda\cdot e_{k+(i-j)}\ra\big)\biggl|\ [\lambda:\mu]\in \P^1\biggr\rb,
w=\pm(e_{k+(i-j)}-e_{a_{i}(\underline{u})}),
\]
of $BR_{i,j}$ is invariant under the induced effective action of the one-dimensional algebraic torus $\C^i/\Ker w$ from the $\T^i$-action \eqref{eq:brij_act} on $BR_{i,j}$. For any of these curves the corresponding weight $w\in\Z^i$ given above is determined up to multiplication by $-1$. This completes the proof.
\end{proof}

One can obtain the hypergraph $\Gamma(BR_{i,j})$ from Propositions \ref{pr:weightsbrij} and \ref{pr:2sphbrij}. The axial function $\alpha(BR_{i,j})$ can be computed from Propositions \ref{pr:weightbr} and \ref{pr:2sphbrij}. Let $\nabla$ be a connection on $(\Gamma(BR_{i,j}),\alpha(BR_{i,j}))$ associated with the action \eqref{eq:brij_act}. We compute the values of $\nabla$ that are necessary for the proof of Theorem \ref{thm:brijnottoric} in the following proposition.

\begin{pr}\label{pr:brijconn}
Let $i,j$ be any integers such that $0\leq j<i$. Let $\underline{u}\in\F^{i}_{2}$ be any vector such that $a_{i}(\underline{u})<i-j$ holds. Then for any integers $k,r=0,\dots,j$ satisfying $k\neq r$, the hypergraph $\Gamma$ has the definite oriented edge \/ $E=E_{\underline{u},k}^{\underline{u},r}$. The connection $\nabla$ is well defined at \/ $E\in E(R(\Gamma))$, and one has the following identities
\[
\nabla_{E} E_{\underline{u},k}^{\underline{u},a}=E_{\underline{u},r}^{\underline{u},a},\
\nabla_{E} E_{\underline{u},k}^{\underline{u}+1_{q},k}=E_{\underline{u},r}^{\underline{u}+\underline{1}_{q},r},\
\nabla_{E} E_{\underline{u},k}^{\underline{u},r}=E_{\underline{u},r}^{\underline{u},k},\
\nabla_{E} E_{\underline{u},k}^{\underline{u}+1_{r+(i-j)},k}=E_{\underline{u},r}^{\underline{u}+1_{k+(i-j)},r},
\]
where $a=0,\dots,j$ and $q=1,\dots,i$ are any integers such that $a\neq k,r$ and $q\neq k+(i-j),r+(i-j)$.
\end{pr}
\begin{proof}
By Proposition \ref{pr:weightsbrij}, the collection of weights at $x_{\underline{u},k}$, as well as at $x_{\underline{u},r}$, is $2$-independent, because $a_{i}(\underline{u})<i-j$. Hence, by Proposition \ref{pr:2sphbrij} there exists the edge $E=E_{\underline{u},k}^{\underline{u},r}$ in the hypergraph $\Gamma$. By Proposition \ref{pr:weightbr} this edge is definite and belongs to the graph $R(\Gamma)$. In order to prove the identities from the claim of the proposition we compute the congruences modulo $\alpha(E)=e_{k+(i-j)}-e_{r+(i-j)}$ between the weights (in particular, vectors in $\Z^i$) in Fig. \ref{fig:weightsbrij}. During the computation we use the identity $b_{q+(i-j)}(\underline{u})=q+(i-j)$ for any integer $q=0,\dots,j$ which holds, because $a_i(\underline{u})<i-j$.
\end{proof}

\begin{figure}
\centering
\caption{Congruences of weights for $\Gamma(BR_{i,j})$. The values of $\alpha$ on the edges from the columns $1,4$ are given in the columns $2,3$ rowwise, respectively. The conditions for the integers $a,q$ are given in the column $5$.}\label{fig:weightsbrij}
{\renewcommand{\arraystretch}{2}
\begin{tabular}{ |l|l|l|l|l| }
  \hline
  \multicolumn{5}{|c|}{$E_{\underline{u},k}^{\underline{u},r}$, $\mod e_{k+(i-j)}-e_{r+(i-j)}$} \\
  \hline
  $E_{\underline{u},k}^{\underline{u},a}$ & $e_{k+(i-j)}-e_{a+(i-j)}$ & $ e_{r+(i-j)}-e_{a+(i-j)}$ & $E_{\underline{u},r}^{\underline{u},a}$ & $a\neq k,r$\\
  \hline
  $E_{\underline{u},k}^{\underline{u}+1_{q},k}$ & $e_{b_{q}(\underline{u})}-e_{a_{q}(\underline{u})}$ & $e_{b_{q}(\underline{u})}-e_{a_{q}(\underline{u})}$ & $E_{\underline{u},r}^{\underline{u}+1_{q},r}$ & $q\neq k+(i-j),r+(i-j)$\\
  \hline
  $E_{\underline{u},k}^{\underline{u},r}$ & $e_{k+(i-j)}-e_{r+(i-j)}$ & $e_{r+(i-j)}-e_{k+(i-j)}$ & $E_{\underline{u},r}^{\underline{u},k}$ & \\
  \hline
  $E_{\underline{u},k}^{\underline{u}+1_{r+(i-j)},k}$ & $e_{r+(i-j)}-e_{a_{i}(\underline{u})}$ & $e_{k+(i-j)}-e_{a_{i}(\underline{u})}$ & $E_{\underline{u},r}^{\underline{u}+1_{k+(i-j)},r}$ & \\
  \hline
\end{tabular}}
\end{figure}

\begin{proof}[Proof of Theorem \ref{thm:brijnottoric}]
For any integers $i,j\geq 0$ which do not satisfy $i>j\geq 2$, the claim of the theorem follows from Proposition \ref{pr:toricbr}. Let $i,j\geq 0$ be any integers such that $i>j\geq 2$. Suppose that $BR_{i,j}$ is a toric variety. The idea of the following argument is to find an invariant $2$-face in $\Gamma$ with a nontrivial action of the monodromy map along it on the external edges. By Proposition \ref{pr:brijconn}, for any integer $k=0,\dots,j-2$ the vertex $x_{\underline{0},k}$ of $\Gamma$ belongs to $V(R(\Gamma))$. Hence, the edge path $\gamma_{k}:=(E_{0,k}^{0,k+1},E_{0,k+1}^{0,k+2},E_{0,k+2}^{0,k})$ belongs to $R(\Gamma)$ for any integer $k=0,\dots,j-2$. This implies that the monodromy map $\Pi_{\gamma_k}$ is well defined for any $k=0,\dots,j-2$. By Proposition \ref{pr:brijconn}, the subgraph $\gamma_{k}$ is a $2$-face of $\Gamma$ for any integer $k=0,\dots,j-2$. By Proposition \ref{pr:brijconn}, we compute $\Pi_{\gamma_k} E_{0,k}^{1_{k+1+(i-j)},k}$ with respect to the connection $\nabla$ as follows.
\[
E_{0,k}^{1_{k+1+(i-j)},k}\mapsto E_{0,k+1}^{1_{k+(i-j)},k+1} \mapsto E_{0,k+2}^{1_{k+(i-j)},k+2} \mapsto E_{0,k}^{1_{k+2+(i-j)},k}.
\]
Hence,
\begin{equation}\label{eq:mon1}
\Pi_{\gamma_{k}} E_{0,k}^{1_{k+1+(i-j)},k}=E_{0,k}^{1_{k+2+(i-j)},k},\ k=0,\dots,j-2.
\end{equation}
It follows from the assumption and Corollary \ref{cor:extassum} that there exists the extension of the $\T^i$-action \eqref{eq:brij_act} on $BR_{i,j}$ to the toric action with the GKM-graph $(\Gamma', \alpha')$ with the connection $\nabla'$. By Proposition \ref{pr:invsubm}, $\gamma_k$ is the subgraph of $\Gamma'$ for any $k=0,\dots,j-2$. Since the edges of $\gamma_{k}$ are definite in $(\Gamma,\alpha)$, one has $\nabla|_{\gamma_{k}}=\nabla'|_{\gamma_{k}}$. In particular, \eqref{eq:mon1} holds with respect to $\nabla'$. However, this contradicts Proposition \ref{pr:monodromy}. The proof is complete.
\end{proof}

\subsection{Ray hypersurface $R_{i,j}$}

In this paragraph we use the notation introduced in \S \ref{ssec:genbrij}. For any integers $i,j$ such that $0\leq j\leq i$, the formula
\begin{equation}\label{eq:sij_act}
(t_{1},\dots,t_{i})\circ(z_i,w_j)=([z_{i,0}:t_1 z_{i,1}:\dots:t_i z_{i,i}],[t_{i-j}^{-1} w_{j,0}:t_{i-j+1}^{-1}w_{j,1}:\dots:t_i^{-1} w_{j,j}]),\ (t_{1},\dots,t_{i})\in\T^i,
\end{equation}
determines a unique effective action of the algebraic torus $\T^i$ on the hypersurface $R_{i,j}$. This follows from the relations \eqref{eq:bfneq} on the homogeneous coordinates on $BF_{i},BF_j$. The hypersurface $R_{i,j}$ is an invariant subvariety of $BF_i\times BF_{j}$ with respect to the action \eqref{eq:sij_act} of the algebraic subtorus $\T^i$ in $\T^i\times\T^j$. Hence, the fixed point set of the $\T^i$-action \eqref{eq:sij_act} on $R_{i,j}$ is the subset of $(BF_i\times BF_{j})^{\T^i\times\T^j}$. It can easily be checked that $R_{i,j}^{\T^i}$ consists of the points $x_{\underline{u},\underline{v}}:=(\C_{\underline{u}},\C_{\underline{v}})\in R_{i,j}$ for any $\underline{u}\in \F_{2}^{i}$ and any $\underline{v}\in \F_{2}^{j}$ such that $a_i(\underline{u})\neq a_j(\underline{v})+(i-j)$. It follows that the open covering of $R_{i,j}$ by the open $\T^i$-invariant subvarieties $(U_{\underline{u}}\times U_{\underline{v}})\cap R_{i,j}$, where $\underline{u}\in\F^i$, $\underline{v}\in \F^j$ are any elements such that $a_i(\underline{u})\neq a_j(\underline{v})+(i-j)$, satisfies the Assumption \ref{asm:torusact}.

\begin{cor}\label{cor:toricr}
Let $n\geq 0$ be any integer.

$(i)$ The variety $R_{0,n+1}$ is a projective toric variety whose moment polytope is combinatorially equivalent to \/ $I^n$. In particular, $R_{0,n+1}$ is a Bott tower;

$(ii)$ The variety $R_{1,n}$ is a projective toric variety whose moment polytope is combinatorially equivalent to the truncation $\cut_{I^{n-2}} I^n$ of \/ $I^n$ at its face \/ $I^{n-2}$;

$(iii)$ The variety $R_{2,2}$ is a projective toric variety whose moment polytope is combinatorially equivalent to the truncation $\cut_{I^{1}} I^3$ of \/ $I^3$ at its edge.
\end{cor}
\begin{proof}
Parts $(i)$ and $(ii)$ follow from Proposition \ref{pr:toricbr}, because $R_{0,n+1}=BR_{0,n+1}$, $R_{1,n}=BR_{1,n}$. Now we prove part $(iii)$. By Theorem \ref{thm:fbr} $(ii)$, there is the algebraic $R_{1,2}$-bundle $R_{2,2}\to \P^1$. This algebraic fiber bundle is represented as the fibered product $R_{2,2}=E\times_{\T^{2}} R_{1,2}\to \P^1$ for some principal algebraic $\T^2$-bundle $E$ over $\P^1$. The equivariant blow-up $R_{1,2}\to BR_{1,2}$ from Remark \ref{rem:br21}, where we identify $R_{1,2}\simeq BR_{2,1}$, induces the $\T^3$-equivariant morphism
\[
\begin{tikzcd}[column sep=5]
\tikzset{
  >/.tip={Stealth[length=3pt, width=4pt, inset=1.8pt]}
}
R_{2,2} \arrow[equal]{r} & E\times_{\T^2} R_{1,2}\arrow{rr}{}\arrow{dr}{} && E\times_{\T^2} BR_{1,2} \arrow{dl}{}\arrow[equal]{r} & BR_{2,2}\\
&& \P^1 &&
\end{tikzcd}
\]
by acting on the fibers. The fan of the toric $\P^1$-bundle $BR_{2,2}\to BF_{2}$ is the normal fan of the polytope in $\R^3$ combinatorially equivalent to the cube $I^3$. The columns of the following matrix
\[
\begin{pmatrix}
1 & -1 & 0 & 0 & 0 & 0\\
0 & 1 & 1 & -1 & 0 & 0\\
0 & -2 & 0 & 1 & 1 & -1\\
\end{pmatrix},
\]
are the generators of the one-dimensional cones for its fan, see \cite{so-17}. Hence, the fan of $R_{2,2}$ is the normal fan of the polytope in $\R^3$ combinatorially equivalent to edge truncation $\cut_{I^1} I^3$ of the cube $I^3$. The columns of the following matrix
\[
\begin{pmatrix}
1 & -1 & 0 & 0 & 0 & 0 & 0\\
0 & 1 & 1 & -1 & 0 & 0 & -1\\
0 & -2 & 0 & 1 & 1 & -1 & 0\\
\end{pmatrix},
\]
are the generators of the one-dimensional cones for its fan. We remark that the last column in the above matrix corresponds to the truncation facet. This completes the proof.
\end{proof}

\begin{rem}
The fan of the toric $R_{1,2}$-bundle $R_{1,3}\to\P^1$ is obtained from the fan of the toric $\P^2$-bundle $BR_{1,3}\to\P^1$ in a similar way as in the proof of Corollary \ref{cor:toricr}. The corresponding map of fibers is the composition of the $\T^2$-equivariant blow-up $R_{1,2}\to BR_{1,2}$ from Remark \ref{rem:br21} and the $\T^2$-equivariant blow-up $BR_{1,2}\to\P^2$ at any fixed point. Hence, the columns of the following matrix
\[
\begin{pmatrix}
1 & -1 & 0 & 0 & 0 & 0 & 0\\
0 & 1 & 1 & -1 & 0 & 0 & -1\\
0 & 0 & 0 & 0 & 1 & -1 & 1\\
\end{pmatrix},
\]
are the generators of the one-dimensional cones for the fan of $R_{1,3}$. We remark that the last column in the above matrix corresponds to the truncation facet of $\cut_{I^1} I^3$.
\end{rem}

For any integer $q=1,\dots,i$ and any $\underline{u}\in\F_{2}^i$ denote by $r(q)=r(\underline{u},q)$ the vector $e_{b_{q}(\underline{u})}-e_{a_{q}(\underline{u})}\in\Z^i$. For any integer $s=1,\dots,j$ and any $\underline{v}\in\F_{2}^j$ denote by $r'(s)=r'(\underline{v},s)$ the vector $e_{a_{s}(\underline{v})+i-j}-e_{b_{s}(\underline{v})+i-j}\in\Z^i$. It is easy to prove the following two propositions.

\begin{pr}\label{pr:weightsrij}
Let $i,j\geq 0$ be any integers such that $i\geq j$. Then for any $\underline{u}\in\F^i$ and any $\underline{v}\in\F^j$ such that $a_i(\underline{u})\neq a_j(\underline{v})+(i-j)$, the weights of the $\T^i$-action \eqref{eq:sij_act} on $R_{i,j}$ at the fixed point $x_{\underline{u},\underline{v}}$ are the elements of the following multiset
\begin{equation}\label{eq:rijweights}
\big\lb r(\underline{u},q)\big|\ q=1,\dots,i\big\rb\cup\big\lb r'(\underline{v},s)\big|\ s=1,\dots,j\big\rb\setminus\big\lb e_{a_{j}(\underline{v})+(i-j)}-e_{a_{i}(\underline{u})}\big\rb.
\end{equation}
\end{pr}

\begin{rem}
If $a_{i}(\underline{u})<a_{j}(\underline{v})+(i-j)$, then $r(q)= e_{a_{j}(\underline{v})+(i-j)}-e_{a_{i}(\underline{u})}$, where $q=a_{j}(\underline{v})+(i-j)$. If $a_{i}(\underline{u})>a_{j}(\underline{v})+(i-j)$, then $r'(s)= e_{a_{j}(\underline{v})+(i-j)}-e_{a_{i}(\underline{u})}$, where $s=a_{i}(\underline{u})-(i-j)$. This justifies the exclusion in \eqref{eq:rijweights}.
\end{rem}

\begin{pr}\label{pr:2linrij}
Let $i,j\geq 0$ be any integers such that $i\geq j$. Let $\underline{u}\in\F^i$, $\underline{v}\in\F^j$ be any vectors satisfying $a_i(\underline{u})\neq a_j(\underline{v})+(i-j)$. Then the multiset of collections of pairwise proportional weights of the $\T^i$-action \ref{eq:sij_act} on $R_{i,j}$ at $x_{\underline{u},\underline{v}}$ consists of the (unordered) pairs $r(q), r'(s)$ of weights, where $q=1,\dots,i$ and $s=1,\dots,j$ are any integers satisfying the following conditions
\begin{equation}\label{eq:conddep}
\lb a_{q}(\underline{u}), b_{q}(\underline{u})\rb=\lb a_{s}(\underline{v})+(i-j), b_{s}(\underline{v})+(i-j)\rb\neq \lb a_{j}(\underline{v})+(i-j),a_{i}(\underline{u})\rb.
\end{equation}
The $\T^i$-invariant subvariety $Y=Y(x_{\underline{u},\underline{v}},[r(q)])$ of \/ $R_{i,j}$ corresponding to the weight $r(q)\in \Z^i$ (see Section \ref{sec:torus_act}) is \/ $\P^1(\underline{u},q)\times \P^1(\underline{v},s)\subseteq BF_{i}\times BF_{j}$. One has $Y^{\T^i}=\lb x_{\underline{u},\underline{v}},\ x_{\underline{u}+1_{q},\underline{v}},\ x_{\underline{u},\underline{v}+1_{s}},\ x_{\underline{u}+1_{q},\underline{v}+1_{s}}\rb$.
\end{pr}

For any integers $i,j\geq 0$ such that $i\geq j$, let $(\Gamma,\alpha)=(\Gamma(R_{i,j}),\alpha(R_{i,j}))$ be the weight hypergraph associated with the $\T^i$-action \eqref{eq:sij_act} on $R_{i,j}$ (notice that the Assumption \ref{asm:torusact} is satisfied for such an action).

\begin{cor}\label{cor:rvert}
Let $i,j$ be any integers such that $i>j\geq 2$. Then $x_{1_{i-1},1_{j}}$, $x_{1_{i-1},1_{j-1}+1_{j}}$, $x_{1_{i-1},0}\in V(R(\Gamma))$.
\end{cor}
\begin{proof}
Let $\underline{u}=1_{i-1}, \underline{v}=1_{j}$. To prove the first claim of the corollary it is enough to check that the condition \eqref{eq:conddep} fails for $x_{\underline{u},\underline{v}}$. Following the notation introduced in Proposition \ref{pr:2linrij}, if $q<i-1$, then $a_{q}(\underline{u})=0<i-j$, so \eqref{eq:conddep} does not hold. If $q=i-1$, then $b_{q}(\underline{u})=b_{i-1}(\underline{u})=0<i-j$, so the condition \eqref{eq:conddep} is not satisfied. If $q=i$, then
\[
\lbrace a_{q}(\underline{u}), b_{q}(\underline{u})\rbrace=\lbrace a_{i}(\underline{u}), b_{i}(\underline{u})\rbrace=\lbrace i,i-1\rbrace=\lb a_{j}(\underline{v})+(i-j),a_{i}(\underline{u})\rbrace.
\]
Hence, the condition \eqref{eq:conddep} is not satisfied in this case, as well. The proof of the second claim from the corollary is obtained by substituting $1_{i-1},1_{j-1}+1_{j}$ for $\underline{u},\underline{v}$ in the above proof, respectively. Now let $\underline{u}=1_{i-1}, \underline{v}=0$. If $q<i-1$, then $a_q(1_{i-1})=0$, and \eqref{eq:conddep} fails. If $q=i-1$, then $b_{i-1}(1_{i-1})=0$, and \eqref{eq:conddep} fails. Let $q=i$, so that $\lbrace a_{q}(\underline{u}), b_{q}(\underline{u})\rbrace=\lb i-1,i\rb$. Then for any $s=1,\dots,j$ one has $a_s(0)+i-j=i-j<i-1$, and \eqref{eq:conddep} fails. The proof is complete.
\end{proof}

\begin{pr}\label{pr:2sphrij}
Let $i,j\geq 0$ be any integers such that $i\geq j$. Let $\underline{u}\in\F_{2}^i$, $\underline{v}\in\F_{2}^j$ be any elements. Then:

$(i)$ For any integer $p=0,\dots, i$ satisfying $a_i(\underline{u}+1_p)\neq a_j(\underline{v})+(i-j)$, the hypergraph $\Gamma$ has a pointed hyperedge $E$ such that $x_{\underline{u},\underline{v}}, x_{\underline{u}+1_{p},\underline{v}}\in E$, and $\alpha(E)=\pm(e_{b_{p}(\underline{u})}-e_{a_p(\underline{u})})$;

$(ii)$ For any integer $q=0,\dots, j$ satisfying $a_i(\underline{u})\neq a_j(\underline{v}+1_{q})+(i-j)$, the hypergraph $\Gamma$ has a pointed hyperedge $E$ such that $x_{\underline{u},\underline{v}}, x_{\underline{u},\underline{v}+1_{q}}\in E$, and $\alpha(E)=\pm(e_{a_q(\underline{v})+(i-j)}-e_{b_{q}(\underline{v})+(i-j)})$;

$(iii)$ If there exist integers $r=1,\dots,i$ and $s=1,\dots,j$ satisfying $a_{i}(\underline{u}+1_{r})=a_{j}(\underline{v})+(i-j)$ and $a_{i}(\underline{u})=a_{j}(\underline{v}+1_{s})+(i-j)$, then the hypergraph $\Gamma$ has a pointed hyperedge $E$ such that $x_{\underline{u},\underline{v}},x_{\underline{u}+1_{r},\underline{v}+1_{s}}\in E$, and $\alpha(E)=\pm(e_{a_{j}(\underline{v})+(i-j)}-e_{a_{i}(\underline{u})})$.
\end{pr}
\begin{proof}
It is easy to prove that any of the following irreducible rational curves
\[
\biggl\lb\big(\lambda\cdot \C_{\underline{u}}+\mu\cdot \C_{1_p},\C_{\underline{v}}\big)\bigg|\ [\lambda:\mu]\in \P^1\biggr\rb,\ w=\pm(e_{b_{p}(\underline{u})}-e_{a_p(\underline{u})}),
\]
\[
\biggl\lb\big(\C_{\underline{u}},\lambda\cdot \C_{\underline{v}}+\mu\cdot \C_{1_q}\big)\bigg|\ [\lambda:\mu]\in \P^1\biggr\rb,\ w=\pm(e_{a_q(\underline{v})+(i-j)}-e_{b_{q}(\underline{v})+(i-j)}),
\]
\[
\biggl\lb\big(\lambda\cdot \C_{\underline{u}}+\mu\cdot \C_{\underline{u}+1_{r}},\mu\cdot \C_{\underline{v}+1_{s}}-\lambda\cdot \C_{\underline{v}}\big)\bigg|\ [\lambda:\mu]\in \P^1\biggr\rb,\ w=\pm(e_{a_{j}(\underline{v})+(i-j)}-e_{a_{i}(\underline{u})}),
\]
of $R_{i,j}$ is invariant under the induced effective action of the one-dimensional algebraic torus $\C^i/\Ker w$ from the $\T^i$-action \eqref{eq:sij_act} on $R_{i,j}$. For any of these curves the corresponding weight $w\in \Z^i$ given above is determined up to multiplication by $-1$. This completes the proof.
\end{proof}

\begin{rem}
The condition from the third case of Proposition \ref{pr:2sphrij} holds iff the numbers $a_{i}(\underline{u})$, $a_{j}(\underline{v})+(i-j)$ belong to the images of the functions $f(r):=a_{j}(\underline{v}+1_{r})+(i-j)$, $g(s):=a_{i}(\underline{u}+1_{s})$, where $r$ runs over $\lbrace 1,\dots,i\rbrace$ and $s$ runs over $\lbrace 1,\dots,j\rbrace$, respectively. If this condition holds, then the number of the weights from first two cases in Proposition \ref{pr:2sphrij} is equal to $i+j-2$, otherwise this number is equal to $i+j-1$.
\end{rem}

One can obtain the hypergraph $\Gamma(R_{i,j})$ from Propositions \ref{pr:2linrij} and \ref{pr:2sphrij}. The axial function $\alpha(R_{i,j})$ can be computed from Propositions \ref{pr:weightsrij} and \ref{pr:2sphrij}. Let $\nabla$ be a connection on $(\Gamma(R_{i,j}),\alpha(R_{i,j}))$ associated the action \eqref{eq:sij_act}. In the following proposition we compute the values of $\nabla$ that are necessary for the proof of Theorem \ref{thm:rijnontoric}.

\begin{figure}
\centering
\caption{Weight hypergraph of $R_{2,2}$. Black edges and grey hyperedges.}\label{fig:5prysm}
\begin{tikzpicture}%
	[x={(0.606396cm, -0.268332cm)},
	y={(0.794802cm, 0.232887cm)},
	z={(-0.023945cm, 0.934752cm)},
	scale=2.000000,
	back/.style={loosely dotted, thick},
	edge/.style={color=white!5!black, thick},
	facet/.style={fill=white, fill opacity=0.300000},
	facetg/.style={fill=gray, fill opacity=0.300000},
	facetb/.style={fill=blue, fill opacity=0.300000},
	facetlast/.style={fill=white!50!black,fill opacity=0.300000},
	vertex/.style={inner sep=1pt,circle,draw=white!5!black,fill=white!5!black,thick,anchor=base},
    vertexr/.style={inner sep=1pt,circle,draw=red!5!black,fill=white!5!black,thick,anchor=base}]	
%
%
\coordinate (1.00000, 0.00000, -0.50000) at (1.00000, 0.00000, -0.50000);
\coordinate (0.30902, 0.95106, -0.50000) at (0.30902, 0.95106, -0.50000);
\coordinate (-0.80902, 0.58779, -0.50000) at (-0.80902, 0.58779, -0.50000);
\coordinate (-0.80902, -0.58779, -0.50000) at (-0.80902, -0.58779, -0.50000);
\coordinate (0.30902, -0.95106, -0.50000) at (0.30902, -0.95106, -0.50000);
\coordinate (1.00000, 0.00000, 0.50000) at (1.00000, 0.00000, 0.50000);
\coordinate (0.30902, 0.95106, 0.50000) at (0.30902, 0.95106, 0.50000);
\coordinate (-0.80902, 0.58779, 0.50000) at (-0.80902, 0.58779, 0.50000);
\coordinate (-0.80902, -0.58779, 0.50000) at (-0.80902, -0.58779, 0.50000);
\coordinate (0.30902, -0.95106, 0.50000) at (0.30902, -0.95106, 0.50000);
\draw[edge] (0.30902, 0.95106, -0.50000) -- (-0.80902, 0.58779, -0.50000);
\draw[edge] (-0.80902, 0.58779, -0.50000) -- (-0.80902, -0.58779, -0.50000);
\draw[edge] (-0.80902, 0.58779, -0.50000) -- (-0.80902, 0.58779, 0.50000);
\node[vertex] at (-0.80902, 0.58779, -0.50000)     {};
\fill[facetg] (0.30902, -0.95106, 0.50000) -- (0.30902, -0.95106, -0.50000) -- (-0.80902, -0.58779, -0.50000) -- (-0.80902, -0.58779, 0.50000) -- cycle {};
\fill[facet] (0.30902, -0.95106, 0.50000) -- (1.00000, 0.00000, 0.50000) -- (0.30902, 0.95106, 0.50000) -- (-0.80902, 0.58779, 0.50000) -- (-0.80902, -0.58779, 0.50000) -- cycle {};
\fill[facet] (0.30902, -0.95106, 0.50000) -- (0.30902, -0.95106, -0.50000) -- (1.00000, 0.00000, -0.50000) -- (1.00000, 0.00000, 0.50000) -- cycle {};
\fill[facetg] (0.30902, 0.95106, 0.50000) -- (0.30902, 0.95106, -0.50000) -- (1.00000, 0.00000, -0.50000) -- (1.00000, 0.00000, 0.50000) -- cycle {};
\draw[edge] (1.00000, 0.00000, -0.50000) -- (0.30902, -0.95106, -0.50000);
\draw[edge] (1.00000, 0.00000, 0.50000) -- (0.30902, -0.95106, 0.50000);
\draw[edge] (0.30902, 0.95106, 0.50000) -- (-0.80902, 0.58779, 0.50000);
\draw[edge] (-0.80902, 0.58779, 0.50000) -- (-0.80902, -0.58779, 0.50000);
\node[vertex] at (1.00000, 0.00000, -0.50000)     {};
\node[vertex] at (0.30902, 0.95106, -0.50000)     {};
\node[vertex] at (-0.80902, -0.58779, -0.50000)     {};
\node[vertex] at (0.30902, -0.95106, -0.50000)     {};
\node[vertex] at (1.00000, 0.00000, 0.50000)     {};
\node[vertex] at (0.30902, 0.95106, 0.50000)     {};
\node[vertex] at (-0.80902, 0.58779, 0.50000)     {};
\node[vertex] at (-0.80902, -0.58779, 0.50000)     {};
\node[vertex] at (0.30902, -0.95106, 0.50000)     {};
\end{tikzpicture} 
\end{figure}

\begin{pr}\label{pr:rijconn}
Let $i,j\geq 0$ be any integers such that $i>j\geq 2$ holds. Then the graph $R(\Gamma)$ has the definite oriented edges $E_{0,1_{j}}^{0,0}, E_{0,0}^{1_{i-1},0}, E_{1_{i-1},0}^{1_{i-1},1_{j}}, E_{1_{i-1},1_{j}}^{1_{i-1},1_{j-1}+1_{j}}$, and the following identities hold.
\[
(i)\
\nabla_{E} E_{0, 1_{j}}^{1_{q},1_{j}}=E_{0,0}^{1_{q},0},\
\nabla_{E} E_{0, 1_{j}}^{0,1_{r}+1_{j}}=E_{0,0}^{0,1_{r}},\
\nabla_{E} E_{0, 1_{j}}^{0,0}=E_{0,0}^{0,1_{j}},\
\nabla_{E} E_{0, 1_{j}}^{1_{i-j},1_{j}}=E_{0,0}^{1_{i},0},
\]
where $q=1,\dots,i$ and $r=1,\dots,j$ are any integers such that $q\neq i-j,i$; $r\neq j$, and $E=E_{0,1_{j}}^{0,0}$;
\[
(ii)\
\nabla_{E} E_{0,0}^{1_{q},0}=E_{1_{i-1},0}^{1_{q}+1_{i-1},0},\
\nabla_{E} E_{0,0}^{0,1_{r}}=E_{1_{i-1},0}^{1_{i-1},1_{r}},\
\nabla_{E} E_{0,0}^{1_{i-1},0}=E_{1_{i-1},0}^{0,0},\
\nabla_{E} E_{0,0}^{0,1_{j-1}}=E_{1_{i-1},0}^{1_{i-j}+1_{i-1},0},\
\nabla_{E} E_{0,0}^{1_{i},0}=E_{1_{i-1},0}^{1_{i-1}+1_{i},0},
\]
where $q=1,\dots,i$ and $r=1,\dots,j$ are any integers such that $q\neq i-j,i-1,i$; $r\neq j-1$, and $E=E_{0,0}^{1_{i-1},0}$;
\[
(iii)\
\nabla_{E} E_{1_{i-1},0}^{1_{q}+1_{i-1},0}=E_{1_{i-1},1_{j}}^{1_{q}+1_{i-1},1_{j}},\
\nabla_{E} E_{1_{i-1},0}^{1_{i-1},1_{r}}=E_{1_{i-1},1_{j}}^{1_{i-1},1_{r}+1_{j}},\
\nabla_{E} E_{1_{i-1},0}^{0,0}=E_{1_{i-1},1_{j}}^{0,1_{j}},\
\nabla_{E} E_{1_{i-1},0}^{1_{i-1}+1_{i},0}=E_{1_{i-1},1_{j}}^{1_{i-1},1_{j-1}+1_{j}},\
\]
\[
\nabla_{E} E_{1_{i-1},0}^{1_{i-1},1_{j}}=E_{1_{i-1},1_{j}}^{1_{i-1},0},
\]
where $q=1,\dots,i$ and $r=1,\dots,j$ are any integers such that $q\neq i-1,i$; $r\neq j-1,j$, and $E=E_{1_{i-1},0}^{1_{i-1},1_{j}}$;
\[
(iv)\
\nabla_{E} E_{1_{i-1},1_{j}}^{1_{q}+1_{i-1},1_{j}}=E_{1_{i-1},1_{j-1}+1_{j}}^{1_{q}+1_{i-1},1_{j-1}+1_{j}},\
\nabla_{E} E_{1_{i-1},1_{j}}^{1_{i-1},1_{r}+1_{j}}=E_{1_{i-1},1_{j-1}+1_{j}}^{1_{i-1},1_{r}+1_{j-1}+1_{j}},\
\nabla_{E} E_{1_{i-1},1_{j}}^{0,1_{j}}=E_{1_{i-1},1_{j-1}+1_{j}}^{0,1_{j-1}+1_{j}},\
\]
\[
\nabla_{E} E_{1_{i-1},1_{j}}^{1_{i-1},1_{j-1}+1_{j}}=E_{1_{i-1},1_{j-1}+1_{j}}^{1_{i-1},1_{j}},\
\nabla_{E} E_{1_{i-1},1_{j}}^{1_{i-1},0}=E_{1_{i-1},1_{j-1}+1_{j}}^{1_{i-1}+1_{i},1_{j-1}},
\]
where $q=1,\dots,i$ and $r=1,\dots,j$ are any integers such that $q\neq i-1,i$; $r\neq j-1,j$, and $E=E_{1_{i-1},1_{j}}^{1_{i-1},1_{j-1}+1_{j}}$.
\end{pr}
\begin{proof}
Notice that the edge $E$ from any of the four cases from the proposition has vertex $x_{\underline{u},\underline{v}}$ for some $\underline{u},\underline{v}$ such that $a_{i}(\underline{u})<i-j$. By Proposition \ref{pr:weightsrij}, $2$-linear independence of weights now follows from definiteness of $E$. To prove the identities from the claim of the proposition we deduce the congruences between the weights of $(\Gamma,\alpha)$ given in Fig. \ref{fig:weightsrij}. (Here we follow the notation introduced in the proof of Proposition \ref{pr:brijconn}).
\end{proof}

\begin{figure}
\centering
\caption{Congruences of weights for $\Gamma(R_{i,j})$ in four tables. The values of $\alpha$ on the edges from the columns $1,4$ are given in the columns $2,3$ rowwise, respectively. The conditions for the integers $r,q$ are given in the column $5$.}\label{fig:weightsrij}
{\renewcommand{\arraystretch}{2}
(i)\
\begin{tabular}{ |l|l|l|l|l| }
  \hline
  \multicolumn{5}{|c|}{$E_{0,1_{j}}^{0,0}$, $\mod e_{i}-e_{i-j}$} \\
  \hline
  $E_{0, 1_{j}}^{1_{q},1_{j}}$ & $e_{q}$ & $ e_{q}$ & $E_{0,0}^{1_{q},0}$ & $q\neq i-j,i$\\
  \hline
  $E_{0, 1_{j}}^{0,1_{r}+1_{j}}$ & $e_{i-j}-e_{r+(i-j)}$ & $e_{i-j}-e_{r+(i-j)}$ & $E_{0,0}^{0,1_{r}}$ & $r\neq j$\\
  \hline
  $E_{0, 1_{j}}^{0,0}$ & $e_{i}-e_{i-j}$ & $e_{i-j}-e_{i}$ & $E_{0,0}^{0,1_{j}}$ & \\
  \hline
  $E_{0, 1_{j}}^{1_{i-j},1_{j}}$ & $e_{i-j}$ & $e_{i}$ & $E_{0,0}^{1_{i},0}$ & \\
  \hline
\end{tabular}
\smallskip

(ii)\
\begin{tabular}{ |l|l|l|l|l| }
  \hline
  \multicolumn{5}{|c|}{$E_{0,0}^{1_{i-1},0}$, $\mod e_{i-1}$} \\
  \hline
  $E_{0,0}^{1_{q},0}$ & $e_{q}$ & $ e_{q}$ & $E_{1_{i-1},0}^{1_{q}+1_{i-1},0}$ & $q\neq i-j,i-1,i$\\
  \hline
  $E_{0,0}^{0,1_{r}}$ & $e_{i-j}-e_{r+(i-j)}$ & $e_{i-j}-e_{r+(i-j)}$ & $E_{1_{i-1},0}^{1_{i-1},1_{r}}$ & $r\neq j-1$\\
  \hline
  $E_{0,0}^{1_{i-1},0}$ & $e_{i-1}$ & $-e_{i-1}$ & $E_{1_{i-1},0}^{0,0}$ & \\
  \hline
  $E_{0,0}^{0,1_{j-1}}$ & $e_{i-j}-e_{i-1}$ & $e_{i-j}$ & $E_{1_{i-1},0}^{1_{i-j}+1_{i-1},0}$ & \\
  \hline
  $E_{0,0}^{1_{i},0}$ & $e_{i}$ & $e_{i}-e_{i-1}$ & $E_{1_{i-1},0}^{1_{i-1}+1_{i},0}$ & \\
  \hline
\end{tabular}
\smallskip

(iii)\
\begin{tabular}{ |l|l|l|l|l| }
  \hline
  \multicolumn{5}{|c|}{$E_{1_{i-1},0}^{1_{i-1},1_{j}}$, $\mod e_{i-j}-e_{i}$} \\
  \hline
  $E_{1_{i-1},0}^{1_{q}+1_{i-1},0}$ & $e_{q}$ & $ e_{q}$ & $E_{1_{i-1},1_{j}}^{1_{q}+1_{i-1},1_{j}}$ & $q\neq i-1,i$\\
  \hline
  $E_{1_{i-1},0}^{1_{i-1},1_{r}}$ & $e_{i-j}-e_{r+(i-j)}$ & $e_{i-j}-e_{r+(i-j)}$ & $E_{1_{i-1},1_{j}}^{1_{i-1},1_{r}+1_{j}}$ & $r\neq j-1,j$\\
  \hline
  $E_{1_{i-1},0}^{0,0}$ & $-e_{i-1}$ & $-e_{i-1}$ & $E_{1_{i-1},1_{j}}^{0,1_{j}}$ & \\
  \hline
  $E_{1_{i-1},0}^{1_{i-1}+1_{i},0}$ & $e_{i}-e_{i-1}$ & $e_{i-j}-e_{i-1}$ & $E_{1_{i-1},1_{j}}^{1_{i-1},1_{j-1}+1_{j}}$ & \\
  \hline
  $E_{1_{i-1},0}^{1_{i-1},1_{j}}$ & $e_{i-j}-e_{i}$ & $e_{i}-e_{i-j}$ & $E_{1_{i-1},1_{j}}^{1_{i-1},0}$ & \\
  \hline
\end{tabular}
\smallskip

(iv)\
\begin{tabular}{ |l|l|l|l|l| }
  \hline
  \multicolumn{5}{|c|}{$E_{1_{i-1},1_{j}}^{1_{i-1},1_{j-1}+1_{j}}$, $\mod e_{i-j}-e_{i-1}$} \\
  \hline
  $E_{1_{i-1},1_{j}}^{1_{q}+1_{i-1},1_{j}}$ & $e_{q}$ & $ e_{q}$ & $E_{1_{i-1},1_{j-1}+1_{j}}^{1_{q}+1_{i-1},1_{j-1}+1_{j}}$ & $q\neq i-1,i$\\
  \hline
  $E_{1_{i-1},1_{j}}^{1_{i-1},1_{r}+1_{j}}$ & $e_{i-j}-e_{r+(i-j)}$ & $e_{i-j}-e_{r+(i-j)}$ & $E_{1_{i-1},1_{j-1}+1_{j}}^{1_{i-1},1_{r}+1_{j-1}+1_{j}}$ & $r\neq j-1,j$\\
  \hline
  $E_{1_{i-1},1_{j}}^{0,1_{j}}$ & $-e_{i-1}$ & $-e_{i-1}$ & $E_{1_{i-1},1_{j-1}+1_{j}}^{0,1_{j-1}+1_{j}}$ & \\
  \hline
  $E_{1_{i-1},1_{j}}^{1_{i-1},1_{j-1}+1_{j}}$ & $e_{i-j}-e_{i-1}$ & $e_{i-1}-e_{i-j}$ & $E_{1_{i-1},1_{j-1}+1_{j}}^{1_{i-1},1_{j}}$ & \\
  \hline
  $E_{1_{i-1},1_{j}}^{1_{i-1},0}$ & $e_{i}-e_{i-j}$ & $e_{i}-e_{i-1}$ & $E_{1_{i-1},1_{j-1}+1_{j}}^{1_{i-1}+1_{i},1_{j-1}}$ & \\
  \hline
\end{tabular}}
\end{figure}

\begin{proof}[Proof of Theorem \ref{thm:rijnontoric}]
For any integers $i,j\geq 0$ such that $\min\lb i,j\rb=0,1$ or $i=j=2$, the claim of the theorem holds by Corollary \ref{cor:toricr}. Let $i,j\geq 0$ be any integers that satisfy neither of these conditions. Without loss of generality, we prove the claim for the case $i> j\geq 2$ only, because $R_{i,j}\simeq R_{j,i}$. Suppose that $R_{i,j}$ is a toric variety. The idea of the following argument is to find a $3$-face in $\Gamma$ with nontrivial action of the monodromy map on the external edges to $G$ along some loops in $G$. By Proposition \ref{pr:2linrij}, for any $\underline{v}\in\F_2^j$ the vertex $x_{0,\underline{v}}$ belongs to $V(R(\Gamma))$, because $a_{q}(\underline{0})=0<i-j$ holds for any $q=1,\dots,i$. Hence, one has $x_{0,1_{j}}, x_{0,0}\in V(R(\Gamma))$. By Corollary \ref{cor:rvert}, one has $x_{1_{i-1},1_{j}}, x_{1_{i-1},1_{j-1}+1_{j}}, x_{1_{i-1},0}\in V(R(\Gamma))$. We conclude that the connection $\nabla$ on $(\Gamma,\alpha)$ is well defined along the edges of the edge path $\gamma=( E_{0,1_{j}}^{0,0},E_{0,0}^{1_{i-1},0},E_{1_{i-1},0}^{1_{i-1},1_{j}})$ as well as along the oriented edge $E_{1_{i-1},1_{j}}^{1_{i-1},1_{j-1}+1_{j}}$ of $G(\Gamma)$.

It follows from the assumption and Corollary \ref{cor:extassum} that there exists an extension of the $\T^i$-action \eqref{eq:sij_act} on $R_{i,j}$ to a toric action with the GKM-graph $(\Gamma',\alpha')$ and the connection $\nabla'$. Let $v$ be any vertex of $\gamma$ or of the oriented edge $E_{1_{i-1},1_{j}}^{1_{i-1},1_{j-1}+1_{j}}$. By Proposition \ref{pr:invsubm}, one has $E_{v}(\Gamma')=E_{v}(\Gamma)$. In particular, $\gamma$ is a subgraph of $\Gamma'$. By Proposition \ref{pr:rijconn}, the subgraph $\gamma$ is invariant in $\Gamma$. The connections $\nabla$, $\nabla'$ coincide along the edges of $\gamma$ as well as along the edge $E_{1_{i-1},1_{j}}^{1_{i-1},1_{j-1}+1_{j}}$ of $G(\Gamma)$ due to definiteness of $\nabla$ at these edges. By Lemma \ref{lm:spanface}, there exists a unique $3$-face $G$ of $\Gamma'$ with respect to $\nabla'$ such that $G$ contains the edges $E_{0,1_{j}}^{0,0},E_{0,1_{j}}^{1_{i-1},1_{j}},E_{0,1_{j}}^{1_{i-j},1_{j}}$. In particular, $\gamma\subset G$ and $E_{0,1_{j}}^{0,1_{j-1}+1_{j}}\not\subset G$. By Lemma \ref{lm:nochord}, this implies that the vertex $x_{0,1_{j-1}+1_{j}}$ of $\Gamma'$ does not belong to $G$. On the other hand, by Proposition \ref{pr:rijconn}, the edges
\[
\Pi_{\gamma} E_{0,1_{j}}^{1_{i-j},1_{j}}=
\nabla_{E_{1_{i-1},0}^{1_{i-1},1_{j}}}\nabla_{E_{0,0}^{1_{i-1},0}}\nabla_{E_{0,1_{j}}^{0,0}} E_{0,1_{j}}^{1_{i-j},1_{j}}=
\nabla_{E_{1_{i-1},0}^{1_{i-1},1_{j}}}\nabla_{E_{0,0}^{1_{i-1},0}} E_{0,0}^{1_{i},0}=
\nabla_{E_{1_{i-1},0}^{1_{i-1},1_{j}}} E_{1_{i-1},0}^{1_{i-1}+1_{i},0}=
E_{1_{i-1},1_{j}}^{1_{i-1},1_{j-1}+1_{j}},
\]
\[
\nabla_{E_{1_{i-1},1_{j}}^{1_{i-1},1_{j-1}+1_{j}}} E_{1_{i-1},1_{j}}^{0,1_{j}}=E_{1_{i-1},1_{j-1}+1_{j}}^{0,1_{j-1}+1_{j}}
\]
belong to $G$. Hence, $x_{0,1_{j-1}+1_{j}}\in E_{1_{i-1},1_{j-1}+1_{j}}^{0,1_{j-1}+1_{j}}$ belongs to $G$. This contradiction proves the theorem.
\end{proof}

\begin{figure}
\centering
\caption{Subgraph of the weight hypergraph $\Gamma=\Gamma(R_{i,j})$ for $i>j\geq 2$}\label{fig:sijsubg}
\begin{tikzpicture}
\begin{scope}[every node/.style={circle,fill,inner sep=0pt, minimum size=6pt,node distance=6pt}]
    \node (1) at (0,-1) [label={[shift={(-0.6,0.0)}]$x_{0,1_{j}}$}]{};{1};
    \node (2) at (2,-1) [label={[shift={(0.0,-0.1)}]below:$x_{0,0}$}]{};{2};
    \node (3) at (2,1) [label={$x_{1_{i-1},0}$}]{};{3};
    \node (4) at (0,1) [label={[shift={(0.0,-0.1)}]$x_{1_{i-1},1_{j}}$}]{};{4};
    \node (5) at (-2,1) [label={[shift={(0.0,-0.4)}]$x_{1_{i-1},1_{j-1}+1_{j}}$}]{};{5};
    \node (6) at (-2,-1) [label={[shift={(0.0,0.3)}]below:$x_{0,1_{j-1}+1_{j}}$}]{};{6};
    \node (7) at (0,-2) [label={[shift={(0.0,0.1)}]below:$x_{1_{i-j},1_{j}}$}]{};{7};
\end{scope}

\begin{scope}[
               ->,>=stealth]
    \draw [black, very thick, ->] plot [smooth, tension=2] coordinates {(0.25,-0.8) (1.75,0) (0.25,0.8)} ;
    \node at (1,0) {$\gamma$};
\end{scope}

\begin{scope}[
              every edge/.style={draw=black,very thick}]
    \path (1) edge (2);
    \path (1) edge (4);
    \path (2) edge (3);
    \path (3) edge (4);
    \path (4) edge (5);
    \path (5) edge (6);
    \path (6) edge (1);
    \path (7) edge (1);
\end{scope}

\end{tikzpicture}
\end{figure}

\section{Acknowledgements}

The author expresses his gratitude to V.M.~Buchstaber and T.E.~Panov for the proposal of problems studied here, many fruitful discussions and constant support while writing this paper. Special thanks are to I.~Arzhantsev, who spotted an error in the early version of the text. It is pleasure to acknowledge several important discussions with A.~Ayzenberg and C.~Shramov.

\appendix

\section{Descriptions of $BR_{i,j}$ and $R_{i,j}$ in terms of blow-ups and fiber bundles}\label{sec:blowup}

Here is the list of main results of this section. In Proposition \ref{pr:brijdef}, for any integers $i,j\geq 1$ we prove that the variety $BR_{i,j}$ is obtained from $H_{i,j}$ by the sequence of $j-1$ blow-ups along strict transforms of the subvarieties of $H_{i,j}$ which are isomorphic to $H_{i,j-k-1}$, where $k$ runs over $\lbrace 1,\dots,j-1\rbrace$. In Theorem \ref{thm:1blowbr}, for any integers $i,j\geq 1$ we prove that $BR_{i,j}$ is the blow-up of $BF_{i-1}\times \P^{j}$ along the subvariety isomorphic to $BR_{i-1,j-1}$. We also find two similar descriptions for Ray hypersurfaces in terms of blow-ups in Proposition \ref{pr:sijdef} and Theorem \ref{thm:1blow}. We find the structures of algebraic fiber bundles on generalised Buchstaber-Ray and Ray hypersurfaces in Theorems \ref{thm:fbbr} and \ref{thm:fbr}, respectively. Throughout of this section we use the notions introduced in Section \ref{sec:def}.

\subsection{Generalised Buchstaber-Ray hypersurface $BR_{i,j}$}

Let $i,j\geq 0$ be any integers. Let $(z,w)=([z_0:z_1:\dots:z_i],[w_0:w_1:\dots:w_j])$ be the homogeneous coordinates of $\P^i\times \P^j$. Denote the subvariety $\lb z_{0}=\cdots=z_{k}=0\rb$ in $\P^i\times\P^j$ by $Z_{k}$ for any $k=1,\dots,i-1$.

\begin{pr}\label{pr:brijdef}
$(i)$ The divisor \/ $BR_{i,j}$ in \/ $BF_{i}\times \P^j$ corresponds to the algebraic line bundle \/ $\beta_{i}^{\vee}\otimes(\eta')^{\vee}$ over \/ $BF_{i}\times \P^j$;

$(ii)$ For any $k=1,\dots,i-1$, two subvarieties $Z_k$ and $\widehat{H}_{i,j}$ (see Definition \ref{defn:Milnor}) intersect transversally in \/ $\P^i\times\P^{j}$. The subvariety $Z_k\cap \widehat{H}_{i,j}$ of \/ $\P^i\times\P^j$ is isomorphic to $H_{i-k-1,j}$.

$(iii)$ The variety $BR_{i,j}$ is a strict transform of \/ $\widehat{H}_{i,j}$ under the sequence of consecutive blow-ups of \/ $\P^i\times\P^{j}$ along strict transforms of the subvarieties \/ $Z_k$ in \/ $\P^i\times\P^{j}$, where $k$ runs over \/ $\lbrace 1,\dots,i-1\rbrace$. In particular, $BR_{i,j}$ is the nonsingular variety that is obtained from $\widehat{H}_{i,j}$ by $i-1$ blow-ups with nonsingular centres.
\end{pr}
\begin{proof}
Under the natural embedding $BF_{i}\to \prod_{r=0}^{i}\P^r$ the restriction of the homogeneous coordinate $z_{i,q}$ to $BF_{i}$ is the global section of the sheaf $\beta_{i}^{\vee}$ over $BF_i$. Hence, the left-hand side of the equation \eqref{eq:brijgen} is a global section of the sheaf $\beta_{i}^{\vee}\otimes(\eta')^{\vee}$ over $BF_{i}\times \P^j$. This proves $(i)$.

Consider the isomorphism
\[
Z_{k}\to\P^{i-k-1}\times \P^j,\ (z,w)\mapsto ([z_{k+1}:\cdots:z_i],[w_0:\cdots:w_j]).
\]
Under this isomorphism, $Z_{k}\cap \widehat{H}_{i,j}$ maps isomorphically to the hypersurface $\widehat{H}_{i-k-1,j}\subset \P^{i-k-1}\times \P^j$. We compute the dimensions as follows.
\[
\dim Z_{k}=i+j-k-1,\ \dim \widehat{H}_{i,j}=i+j-1,\ \dim \widehat{H}_{i-k-1,j}=i-k-1+j-1.
\]
We obtain
\[
\dim Z_{k} + \dim \widehat{H}_{i,j} - \dim \widehat{H}_{i-k-1,j}=\dim(\P^i\times\P^j),
\]
which proves $(ii)$. The projection
\begin{equation}\label{eq:projbfp}
BF_{i}\times \P^j\to \P^i\times\P^j,\ (l_0,\dots,l_{i},l')\mapsto (l_{i},l'),
\end{equation}
decomposes into a sequence of blow-ups along strict transforms of $Z_{k}$, where $k$ runs over $\lbrace 1,\dots,i-1\rbrace$. The subvarieties $Z_k$ and $\widehat{H}_{i,j}$ intersect transversally in $\P^i\times \P^j$ by $(ii)$ for any $k=1,\dots,i-1$. Hence, the argument from \cite[pp. 604--605]{gr-ha-78} applies. Therefore, the restriction of the projection \eqref{eq:projbfp} to $BR_{i,j}$ decomposes into a sequence of blow-ups of $\widehat{H}_{i,j}$. This proves $(iii)$.
\end{proof}

\begin{rem}
The arrangement $\lb Z_{k}\cap \widehat{H}_{i,j}|\ k=1,\dots,i-1\rb$ in $\widehat{H}_{i,j}$ is a simple instance of a building set in terminology of \cite{li-09}, because the elements of this arrangement form a chain of embeddings of submanifolds in $\widehat{H}_{i,j}$. The wonderful compactification of this arrangement is isomorphic to the iterated blow-up of $Z_{k}\cap \widehat{H}_{i,j}$, where $k$ runs over $1,\dots,i-1$. This can be seen either directly from the embedding of a blow-up to the Cartesian product, or from \cite[Theorem 1.3, p.537]{li-09}. Comparing this with Proposition \ref{pr:brijdef} $(iii)$, one obtains that $BR_{i,j}$ is a wonderful compactification. The embedding of $BR_{i,j}$ obtained from wonderful compactification is to $\prod_{k=1}^{i-1} \Bl_{Z_{k}\cap \widehat{H}_{i,j}} \widehat{H}_{i,j}$, where the blow-up centres are described by Proposition \ref{pr:brijdef} $(ii)$. In this paper we utilize the different embedding $BR_{i,j}\to BF_i\times\P^j\to (\prod_{k=0}^{i} \P^k)\times \P^j$.
\end{rem}

\begin{rem}
The projective toric variety $BR_{2,2}$ is a Bott tower (see \cite{bu-ra-98} or \cite[p.769, Proposition 9]{so-17}). One can easily compute the fan $\Sigma$ in $\R^3$ of the toric variety $BR_{2,2}$ by following the general description of the fan for any Bott tower (see \cite[p.290, Corollary 7.8.7]{bu-pa-15}). The columns of the following matrix
\[
\begin{pmatrix}
1 & 0 & 0 & -1 & 0 & 0\\
0 & 1 & 0 & 1 & -1 & 0\\
0 & 0 & 1 & -2 & 1 & -1
\end{pmatrix},
\]
are the generators of the respective one-dimensional cones $\sigma_{0,1},\sigma_{0,2},\sigma_{0,3},\sigma_{1,1},\sigma_{1,2},\sigma_{1,3}$ of $\Sigma$. The three-dimensional cones of $\Sigma$ are
\[
\lb\R\la \sigma_{u_1,1},\sigma_{u_2,2},\sigma_{u_3,3}\ra|\ (u_1,u_2,u_3)\in \F_{2}^3\rb.
\]
\end{rem}

Let $i,j\geq 1$ be any integers. Denote by $E$ the subvariety in $BF_{i}\times \P^{j}$ consisting of all points $(l_{0},\dots,l_{i},l')$ such that $l_{i-1}\perp \overline{l'}$ and $l'\perp \C_{\max\lb i,j\rb}$ hold. Each of the conditions $l'\perp \C_{\max\lb i,j\rb}$ and $\overline{l'}\perp \C_{\max\lb i,j\rb}$ is equivalent to the condition $w_j=0$. Therefore, $l_{i}\perp \overline{l'}$ holds for any point from $E$. In particular, $E\subset BR_{i,j}$. Denote the natural embedding $BR_{i,j}\to BF_{i}\times \P^j$ from Definition \ref{def:brij} by $f_{i,j}$. Consider the projection
\[
\pi\colon BR_{i,j}\to BF_{i-1}\times \P^j,\ (l_0,\dots,l_i,l')\mapsto (l_0,\dots,l_{i-1},l').
\]
Consider the embedding $g\colon BF_{i-1}\times \P^{j-1}\to BF_{i-1}\times \P^j$ induced by the identity map on $BF_{i-1}$ and by the embedding $\P^{j-1}\to\P^{j}$ given by $[w_{0}:\dots:w_{j-1}]\mapsto [w_{0}:\dots:w_{j-1}:0]$.

\begin{thm}\label{thm:1blowbr}
$(i)$ The normal bundle of the embedding $g\circ f_{i-1,j-1}: BR_{i-1,j-1}\to BF_{i-1}\times \P^j$ is \/ $(\beta_{i-1}^{\vee}\oplus\underline{\C})\otimes(\eta')^{\vee}$;

\smallskip
\noindent $(ii)$ The morphism $\pi\colon BR_{i,j}\to BF_{i-1}\times \P^j$ is the blow-up of \/ $BF_{i-1}\times \P^j$ along $g\circ f_{i-1,j-1}(BR_{i-1,j-1})$ with the exceptional divisor $E$;

\smallskip
\noindent $(iii)$ The exceptional divisor \/ $E$ is the total space of the algebraic fiber bundle \/ $\P(\beta_{i-1}\oplus\underline{\C})\to BR_{i-1,j-1}$.

\end{thm}
\begin{proof}
The normal bundle of the embedding $g$ is isomorphic to $(\eta')^{\vee}$. The normal bundle of the embedding $f_{i-1,j-1}$ is $\beta_{i-1}^{\vee}\otimes(\eta')^{\vee}$ by Proposition \ref{pr:brijdef} $(i)$. This proves $(i)$. The subvariety $g\circ f_{i-1,j-1}(BR_{i-1,j-1})$ of $BF_{i-1}\times\P^j$ is given by the equations $\lb s_{1}=s_{2}=0\rb$, where
\[
s_1:=\sum_{k=1}^{\min\lb i,j\rb} z_{i-1,i-k}w_{j-k},\ s_{2}:=w_{j},
\]
are global sections of the algebraic line bundles $\beta_{i-1}^{\vee}\otimes(\eta')^{\vee}$ and $(\eta')^{\vee}$ over $BF_{i-1}\times\P^j$, respectively. The regular morphism $\pi$ is an isomorphism outside the zero locus $\lb s_{1}=s_{2}=0\rb$. The restriction of the morphism $\pi$ to the preimage of $\lb s_{1}=s_{2}=0\rb$ is $E=\P(\beta_{i-1}\oplus\underline{\C})$. Since this projective bundle is isomorphic to the projectivisation of the normal bundle of $g\circ f_{i-1,j-1}(BR_{i-1,j-1})$ in $ BF_{i-1}\times\P^j$ by basic property of a blow-up, we conclude that $\pi$ is the required blow-up. Hence, $E$ is the exceptional divisor of the blow-up $\pi$. This proves $(ii)$ and $(iii)$.
\end{proof}

\begin{thm}\label{thm:fbbr}
$(i)$ Let $i,j\geq 0$ be any integers such that $i\geq j+1$. Then the morphism
\[
p\colon BR_{i,j}\to BF_{i-j-1},\ (l_{0},\dots,l_{i},l')\mapsto (l_0,l_{1},\dots,l_{i-j-1}),
\]
is an algebraic $BR_{j+1,j}$-bundle;

$(ii)$ Let $i\geq 1$ be any integer. Then the morphism
\[
p_{1}\colon BR_{i,i}\to \P^1,\ (l_{0},\dots,l_{i},l')\mapsto (l_0,l_{1}),
\]
is an algebraic $BR_{i-1,i}$-bundle.
\end{thm}
\begin{proof}
We prove the claims of this theorem by constructing the trivialisations for the corresponding algebraic fiber bundles. Let $L=(l_{0},\dots,l_{i},l')\in BR_{i,j}$. Recall that any $N=(l_0,\dots,l_{i-j-1})\in BF_{i-j-1}$ is determined by the tuple $z=(z_0,\dots,z_{i-j-1})$ of the homogeneous coordinates. Let $U_{k}=\lbrace z_{i-j-1,k}\neq 0\rbrace$ be the open subvariety of $BF_{i-j-1}$, where $k=0,\dots,i-j-1$. For any $k=0,\dots,i-j-1$ there exists a unique morphism $A_k\colon U_k\times\C^{i+1}\to \C^{i+1}$ such that
\begin{itemize}
\item $A_{k}(N,-)\colon \C^{i+1}\to \C^{i+1}$ is a $\C$-linear map for any $N\in U_k$;\\
\item $A_{k}(N,-)\colon \C^{i+1}\to \C^{i+1}$ takes $e_0,\dots,\widehat{e_k},\dots,e_{i-j-1}$ to $e_{j+2},\dots,e_i$, and takes $e_{i-j},\dots,e_{i}$ to $e_{1},\dots,e_{j+1}$, respectively, for any $N\in U_k$;\\
\item the following (well-defined) conditions (where $z_{i-j-1}=z_{i-j-1}(N)$ is a tuple of homogeneous coordinates) hold:
\[
A_{k}\bigl(N,\frac{1}{z_{i-j-1,k}}(z_{i-j-1,0}\cdot e_{0}+\cdots +z_{i-j-1,i-j-1}\cdot e_{i-j-1})\bigr)=e_{0},\ \forall N\in U_k.
\]
\end{itemize}

The desired trivialisation for the fiber bundle $p$ is
\[
p^{-1}(U_k)\to BR_{j+1,j}\times U_{k},\ L\mapsto \biggl(\bigl(A_k(p(L),l_{i-j-1}),\dots,A_k(p(L),l_{i}),l'\bigr),p(L)\biggr),\ k=0,\dots,i-j-1.
\]

$(ii)$ Let $z=(z_0,z_1)$ be the tuple of homogeneous coordinates of any $N\in BF_1$. Let $V_{k}=\lbrace z_{1,k}\neq 0\rbrace$ be the open subvariety of $BF_{1}$, where $k=0,1$. For any $k=0,1$ let $B_k\colon V_k\times\C^{i+1}\to \C^{i+1}$ be a morphism such that for any $N\in V_k$ the $\C$-linear map $B_{k}(N,-)\colon \C^{i+1}\to \C^{i+1}$ takes $e_2,\dots,e_i$ to $e_1,\dots,e_{i-1}$ and takes $\lbrace e_0,e_1\rbrace\setminus e_{k}$ to $e_1$, respectively. Furthermore, for any $k=0,1$ there exists a unique $B_k$ satisfying the well-defined conditions
\[
B_{k}\bigl(N,\frac{1}{z_{1,k}}(z_{1,0}e_{0}+z_{1,1}e_{1})\bigr)=e_0,\ \forall N\in V_k.
\]
The desired trivialisation for the fiber bundle $p_1$ is
\[
p^{-1}_{1}(U_k)\to BR_{i-1,i}\times V_{k},\ L\mapsto \biggl(\bigl(B_k(p_1(L),l_{1}),\dots,B_k(p_1(L),l_{i}),l'\bigr),p_1(L)\biggr),\ k=0,1.
\]
\end{proof}

\subsection{Ray variety $R_{i,j}$}

Let $i,j\geq 0$ be any integers. Denote the subvariety $\lb w_{0}=\cdots=w_{k}=0\rb$ in $BF_{i}\times\P^{j}$ by $W_{k}$ for any $k=1,\dots,j-1$. The proof of the following proposition is similar to the proof of Proposition \ref{pr:brijdef}.

\begin{pr}\label{pr:sijdef}
$(i)$ The divisor \/ $R_{i,j}$ corresponds to the algebraic line bundle \/ $\beta_{i}^{\vee}\otimes(\beta'_{j})^{\vee}$ over $BF_{i}\times BF_{j}$;

$(ii)$ For any $k=1,\dots,j-1$, two subvarieties $W_k$ and $BR_{i,j}$ intersect transversally in \/ $BF_{i}\times\P^{j}$. The subvariety $W_k \cap BR_{i,j}$ of \/ $BF_{i}\times \P^j$ is isomorphic to $BR_{i,j-1-k}$.

$(iii)$ The variety $R_{i,j}$ is a strict transform of \/ $BR_{i,j}$ under the sequence of consecutive blow-ups of \/ $BF_{i}\times\P^{j}$ along strict transforms of the subvarieties $W_k$ in \/ $BF_{i}\times\P^{j}$, where $k$ runs over $\lbrace 1,\dots,j-1\rbrace$. In particular, $R_{i,j}$ is the nonsingular projective variety that is obtained from \/ $BR_{i,j}$ by $j-1$ blow-ups with nonsingular centres.
\end{pr}

\begin{rem}
The arrangement $\lb W_{k}\cap BR_{i,j}|\ k=1,\dots,j-1\rb$ in $BR_{i,j}$ is also a building set in terminology of \cite{li-09}, because the elements of this arrangement form a chain of embeddings of submanifolds in $BR_{i,j}$. The wonderful compactification of this arrangement is isomorphic to the iterated blow-up of $W_{k}\cap BR_{i,j}$, where $k$ runs over $1,\dots,j-1$. This can be seen either directly from the embedding of a blow-up to the Cartesian product, or from \cite[Theorem 1.3, p.537]{li-09}. Comparing this with Proposition \ref{pr:sijdef} $(iii)$, one obtains that $R_{i,j}$ is a wonderful compactification. The embedding of $R_{i,j}$ obtained from wonderful compactification is to $\prod_{k=1}^{j-1} \Bl_{W_{k}\cap BR_{i,j}} BR_{i,j}$, where the blow-up centres are described by Proposition \ref{pr:sijdef} $(ii)$. In this paper we utilize the different embedding $R_{i,j}\to BF_i\times BF_j\to (\prod_{k=0}^{i} \P^k)\times (\prod_{q=0}^{j} \P^q)$.
\end{rem}

Let $i,j\geq 1$ be any integers. Denote by $D$, $D_1$, $D_2$ the subvarieties in $BF_{i}\times BF_{j}$ consisting of all points $(l_{0},\dots,l_{i},l'_0,\dots,l'_j)$ such that $l_{i-1}\perp \overline{l'_{j-1}}$ and $l_{i}\perp \overline{l'_{j}}$ hold; $l_{i}\perp \overline{l'_j}$ and $l_i=l_{i-1}$ hold; $l_{i}\perp \overline{l'_j}$ and $l'_j=l'_{j-1}$ hold, respectively. It is straight-forward to prove the following lemma.

\begin{lm}\label{lm:diveqs}
One has $D=D_{1}\cup D_{2}$, where $D_{1}, D_{2}\subset R_{i,j}$ are nonsingular irreducible hypersurfaces of \/ $R_{i,j}$. The intersection $D_1\cap D_{2}$ is isomorphic to $R_{i-1,j-1}$.
\end{lm}

Denote by $r_{i,j}$ the natural embedding $R_{i,j}\to BF_{i}\times BF_{j}$ from Definition \ref{def:sij}. Consider the following morphisms
\[
\pi_{1}\colon R_{i,j}\to BF_{i-1}\times BF_{j},\ (l_{0},\dots,l_{i-1},l_{i},l'_0,\dots,l'_j)\mapsto (l_{0},\dots,l_{i-1},l'_0,\dots,l'_j),
\]
\[
\pi_{2}\colon R_{i,j}\to BF_{i}\times BF_{j-1},\ (l_{0},\dots,l_{i},l'_0,\dots,l'_{j-1},l'_j)\mapsto (l_{0},\dots,l_i,l'_0,\dots,l'_{j-1}),
\]
\[
g_{1}\colon BF_{i-1}\times BF_{j-1}\to BF_{i-1}\times BF_{j},\ (l_{0},\dots,l_{i-1},l'_0,\dots,l'_{j-1})\mapsto (l_{0},\dots,l_{i-1},l'_0,\dots,l'_{j-1},l'_{j-1}),
\]
\[
g_{2}\colon BF_{i-1}\times BF_{j-1}\to BF_{i}\times BF_{j-1},\ (l_{0},\dots,l_{i-1},l'_0,\dots,l'_{j-1})\mapsto (l_{0},\dots,l_{i-1},l_{i-1},l'_0,\dots,l'_{j-1}).
\]

The proof of the following theorem is similar to the proof of Theorem \ref{thm:1blowbr}

\begin{thm}\label{thm:1blow}
$(i)$ The normal bundles of the embeddings
\[
g_1\circ r_{i-1,j-1}\colon R_{i-1,j-1}\to BF_{i-1}\times BF_{j},\ g_2\circ r_{i-1,j-1}\colon R_{i-1,j-1}\to BF_{i}\times BF_{j-1},
\]
are \/ $(\beta_{i-1}^{\vee}\oplus\underline{\C})\otimes(\beta'_{j-1})^{\vee}$ and \/ $((\beta'_{j-1})^{\vee}\oplus\underline{\C})\otimes\beta_{i-1}^{\vee}$, respectively;

\smallskip
\noindent $(ii)$ The morphisms $\pi_{1}\colon R_{i,j}\to BF_{i-1}\times BF_{j}$ and \/ $\pi_{2}\colon R_{i,j}\to BF_{i}\times BF_{j-1}$ are the blow-ups along the centres $g_1\circ r_{i-1,j-1}(R_{i-1,j-1})$ and $g_2\circ r_{i-1,j-1}(R_{i-1,j-1})$ with exceptional divisors $D_1$ and $D_2$, respectively;

\smallskip
\noindent $(iii)$ The exceptional divisors $D_1$ and $D_2$ are the total spaces of the algebraic fiber bundles \/ $\P(\beta_{i-1}\oplus\underline{\C})\to R_{i-1,j-1}$ and \/ $\P(\beta'_{j-1}\oplus\underline{\C})\to R_{i-1,j-1}$, respectively.
\end{thm}

The proof of the following theorem is similar to the proof of Theorem \ref{thm:fbbr}.

\begin{thm}\label{thm:fbr}

$(i)$ Let $i,j\geq 0$ be any integers such that $i\geq j+1$. Then the following morphism
\[
p\colon R_{i,j}\to BF_{i-j-1},\ (l_{0},\dots,l_{i},l'_{0},\dots, l'_{j})\mapsto (l_{0},\dots,l_{i-j-1}),
\]
is an algebraic $R_{j+1,j}$-bundle.

$(ii)$ Let $i\geq 1$ be any integer. Then the following morphisms
\[
p_{1}\colon R_{i,i}\to \P^1,\ (l_{0},\dots,l_{i},l'_{0},\dots, l'_{j})\mapsto (l_{0},l_{1})\mid p_{2}\colon R_{i,i}\to \P^1,\ (l_{0},\dots,l_{i},l'_{0},\dots, l'_{j})\mapsto (l'_{0},l'_{1}),
\]
are algebraic fiber bundles with fibers $R_{i-1,i}$ and $R_{i,i-1}$, respectively.
\end{thm}

\section{Cohomology rings of $BR_{i,j}$ and $R_{i,j}$}\label{sec:coh}

In this section, we prove that the cohomology rings of the hypersurfaces $BR_{i,j}$ and $R_{i,j}$ are isomorphic to the quotients of the known cohomology rings of the ambient varieties $BF_{i}\times \P^{j}$ and $BF_{i}\times BF_{j}$ by the annihilator ideals of the first Chern classes of the respective normal line bundles for any integers $i,j\geq 0$. We deduce the formulas for the Hodge-Deligne polynomials of the hypersurfaces $BR_{i,j}$ and $R_{i,j}$ from the Hodge-Deligne polynomial of $H_{i,j}$ by using the blow-up descriptions of $BR_{i,j}$ and $R_{i,j}$ from Section \ref{sec:blowup}. In particular, we compute all Betti numbers of $BR_{i,j}$ and $R_{i,j}$ for any integers $i,j\geq 0$. In the following, by omitting the coefficient group in singular cohomology we assume $\Z$-coefficients.

\subsection{Cohomology ring of the blow-up of a complex manifold along a submanifold}
Let $\iota\colon Z\subset X$ be any holomorphic embedding of complex compact connected manifolds. Consider the blow-up $\pi\colon \Bl_{Z} X\to X$ of $X$ along $Z$. The exceptional divisor $E$ of $\pi$ is the holomorphic fiber bundle $E\simeq \P(\nu)\to Z$, where the projection map is given by the restriction of $\pi$ to $E$, and $\nu\to Z$ is the normal bundle of $\iota$. The restriction of the projection map $\pi$ to $E$ induces the structure of a $H^{*}(Z\mid \Z)$-module on $H^{*}(\P(\nu)\mid \Z)$.

\begin{thm}[{Leray, Hirsch, see \cite[\S 15]{bo-hi-58}}]\label{thm:cohom_pf}
Let $\xi\to B$ be a complex vector bundle of rank $k$ over $B$. Consider the fiberwise projectivisation $p\colon \mathbb P(\xi)\to B$ of \/ $\xi$. Let $v=c_1(\zeta^{\vee})\in H^2(\mathbb P(\xi);\Z)$ be the first Chern class of the dual to the tautological line bundle $\zeta$ over \/ $\P(\xi)$. Then the following rings
\begin{equation}\label{eq:leray-hirsch}
H^{*}(\mathbb{P}(\xi))\simeq H^{*}(B)[v]/(v^{k}+v^{k-1}c_{1}(\xi)+\dots+c_{k}(\xi)),
\end{equation}
are isomorphic. In particular, $H^{*}(\mathbb{P}(\xi))$ is a free $H^{*}(B)$-module with generators $1,v,\dots,v^{k-1}$.
\end{thm}

\begin{ex}
By applying Theorem \ref{thm:cohom_pf} recurrently to the $\P^1$-bundle $BF_{n}=\P(\beta_{n-1}\oplus\underline{\C})\to BF_{n-1}$, one obtains an isomorphism
\begin{equation}\label{eq:bfncohom}
H^*(BF_{n};\Z)\simeq\frac{\Z[x_1,\dots,x_n]}{(x_q^2-x_q x_{q-1}|\ q=1,\dots,n)},
\end{equation}
of graded rings, where $x_0:=0$, see \cite{bu-pa-15}.
\end{ex}

Let $Y$ be any compact complex submanifold of $X$ which intersects $Z$ transversally in $X$.

\begin{pr}[{\cite{gr-ha-78}, \cite{ha-ba-09}}]\label{pr:blow-up}
The normal bundle of the hypersurface $E$ in $\Bl_{Z} X$ is isomorphic to the tautological line bundle $\zeta\to \P(\nu)$. The strict transform $\widetilde{Y}$ of \/ $Y$ under $\pi$ is isomorphic to $\Bl_{Z\cap Y} Y$. The following abelian groups
\begin{equation}\label{eq:cohgr}
H^{*}(\Bl_{Z} X;\Z)\simeq H^{*}(X;\Z) \oplus H^{*}(\P(\nu)\mid \Z)/H^{*}(Z;\Z)\simeq H^{*}(X;\Z) \oplus H^{*}(Z;\Z)\la v,v^2,\dots,v^{k-1}\ra,
\end{equation}
are naturally isomorphic, where $k$ is the codimension of \/ $Z$ in $X$. The ring $H^{*}(\Bl_{Z} X;\Z)$ is isomorphic to the quotient of the ring on the right hand side of \eqref{eq:cohgr} by the relations
\[
x\cdot v=\iota^*(x)\cdot v,\ x\in H^{*}(X;\Z),
\]
\[
v^{k}+v^{k-1}c_{1}(\nu)+\dots+vc_{k-1}(\nu)+\omega_{X}=0,
\]
where $\omega_{X}\in H^{2k}(X;\Z)$ is Poincar\'e dual to the homology class \/ $\iota_*[Z]\in H_{2(n-k)}(X;\Z)$, and $v$ restricts to $c_1(\zeta^{\vee})$.
\end{pr}

\subsection{Cohomology ring of a hypersurface}

Let $X^n$ be any compact complex manifold with no torsion in $H^*(X;\Z)$. By the Poincar\'e duality, the $\Z$-bilinear form $Q_{X}$,
\begin{equation}\label{eq:bilin}
H^{*}(X;\Z)\times H^{*}(X;\Z)\to\Z,\ Q_{X}(a,b):=\la a\cdot b,[X]\ra,
\end{equation}
given by the natural pairing with the fundamental class $[X]\in H_{2n}(X;\Z)$, is nondegenerate.

In addition, let $X^{n}$ be connected and simply connected. Then the group $H^2(X;\Z)$ is isomorphic to the Picard group of equivalence classes of the holomorphic line bundles over $X$ modulo holomorphic isomorphisms. Let $\xi\to X$ be any holomorphic line bundle. In the following, we assume that the divisor corresponding to $\xi$ is represented by an irreducible nonsingular hypersurface $D$ in $X$. In this case, the homology class of $D$ in $H_{2(n-1)}(X;\Z)$ is Poincar\'e dual to $x=c_{1}(\xi)\in H^2(X;\Z)$. Consider the homomorphism $\iota^*: H^*(X;\Z)\to H^*(D;\Z)$ induced by the natural embedding $\iota\colon D\to X$.

\begin{pr}\label{pr:kerdual}
Suppose that all odd cohomology groups \/ $H^{2k+1}(X;\Z)$ vanish and that \/ $\iota^*$ is an epimorphism. Then \/ $\Ker \iota^{*}$ is the annihilator ideal $\Ann x$ of \/ $x$ in the ring \/ $H^*(X;\Z)$. In particular, the quotient homomorphism $H^*(X;\Z)/ \Ann x \to H^*(D;\Z)$ induced by \/ $\iota^*$ is an isomorphism of rings.
\end{pr}
\begin{proof}
Since $\iota^*$ is an epimorphism, we conclude from $H^{2k+1}(X;\Z)=0$ that $H^{2k+1}(D;\Z)=0$ holds for any integer $k\geq 0$. The universal coefficients formula then implies that the groups $H^*(X;\Z)$,$H^*(D;\Z)$ have no torsion. The class $\iota_{*}(D)\in H_{2(n-1)}(X;\Z)$ is Poincar\'e dual to $x\in H^{2}(X;\Z)$. This means that the identity
\begin{equation}\label{eq:pd}
\la y, \iota_{*}[D]\ra=\la x\cdot y, [X]\ra,
\end{equation}
holds for any $y\in H^{2(n-1)}(X;\Z)$. For any elements $\alpha,\beta\in H^{*}(X;\Z)$ of degree $2k$ and $2(n-k-1)$, respectively, we deduce the following identities
\begin{equation}\label{eq:projfla}
\la \iota^*(\alpha)\iota^*(\beta),[D]\ra=\la \iota^*(\alpha \beta),[D]\ra=\la \alpha\beta, \iota_*[D]\ra=\la x\cdot\alpha \beta,[X]\ra=\la(\alpha x)\cdot \beta,[X]\ra,
\end{equation}
from \eqref{eq:pd}. Let $\alpha$ be any element of $\Ker \iota^*$. Then the left hand side of \eqref{eq:projfla} is zero. Hence, $\alpha x$ belongs to the kernel of the bilinear form $Q_{X}$. Then $\alpha x=0$, because the bilinear form $Q_X$ is nondegenerate. We conclude that $\Ker \iota^*\subseteq \Ann x$.

Let $\alpha\in \Ann x$ be any element. Then the right-hand side of \eqref{eq:projfla} is zero for any $\beta\in H^*(X;\Z)$. We conclude from \eqref{eq:projfla} that $\la \iota^*(\alpha)\widetilde{\beta},[D]\ra=0$ for any  $\widetilde{\beta}\in H^*(D\mid \Z)$, because $\iota^*$ is epimorphic. Hence, $\iota^*(\alpha)$ belongs to the kernel of the bilinear form $Q_D$. We conclude that $\iota^*(\alpha)=0$, because $Q_D$ is nondegenerate. This implies that $\Ann x\subseteq \Ker \iota^*$ holds. The proof is complete.
\end{proof}

In general, the embedding of a hypersurface to the ambient manifold does not induce epimorphism of the respective cohomology groups.

\begin{ex}
For any integers $n,d>0$, let $f_{d}\colon X_d\subset \P^n$ be the embedding of a generic hypersurface $X_d$ of degree $d$ to $\P^n$. One can check that for any even $n$ and any integer $d>2$ the group $H^{n-1}(X_d;\R)$ is nonzero and the homomorphism $f_{d}^{*}$ is not epimorhic. For $d=2$ and $n=3$ the Veronese embedding $f_{2}\colon \P^1\times \P^1\to\P^3$ of the nonsingular quadric induces the homomorphism
\[
f^*_2\colon \Z[x]/(x^4)\to \Z[y,z]/(y^2,z^2),\ x\mapsto y+z,
\]
of the respective cohomology rings, which is clearly not onto. (The last example was pointed out to the author by A. Ayzenberg.)
\end{ex}

\begin{lm}\label{lm:emb}
Let $\xi,\nu$ be complex vector bundles over a compact topological space $B$. Suppose that $\nu$ is a subbundle of \/ $\xi$. Let $\alpha\colon \P(\nu)\to\P(\xi)$ be the corresponding embedding. Then the induced homomorphism $\alpha^*\colon H^*(\P(\xi);\Z)\to H^*(\P(\nu);\Z)$ is onto.
\end{lm}
\begin{proof}
Consider the tautological line bundles $\zeta\to\P(\nu)$, $\zeta'\to \P(\xi)$ of the respective projective fiber bundles. Let $k=\rk \nu$, $r=\rk \xi$. By Theorem \ref{thm:cohom_pf}, the following free $H^*(B)$-modules
\begin{equation}\label{eq:fmod}
H^*(\P(\nu))=H^*(B)\la 1,u,\dots,u^{k-1}\ra,\ H^*(\P(\xi))=H^*(B)\la 1,v,\dots,v^{r-1}\ra,
\end{equation}
are isomorphic, where $u=c_1(\zeta^{\vee})$, $v=c_1((\zeta')^{\vee})$. By the definition, $\alpha^*(\zeta')=\zeta'|_{\alpha(\P(\nu))}=\zeta$. Hence, $\alpha^*(v)=u$. Now the statement follows from \eqref{eq:fmod}, because $k\leq r$.
\end{proof}

Recall that $Z$ is a submanifold and $D$ is a hypersurface in $X$. Assume that $Z$ and $D$ intersect transversally in $X$. Then by Proposition \ref{pr:blow-up}, the strict transform $\widetilde{D}$ of $D$ with respect to the blow-up $\Bl_{Z} X\to X$ is isomorphic to $\Bl_{Z\cap D} D$. Let $\widetilde{\iota}\colon \widetilde{D} \to \widetilde{X}\simeq \Bl_{Z} X$ be the corresponding embedding.

\begin{lm}\label{lm:blow_div}
Suppose that the embeddings $D\to X$ and $Z\cap D\to Z$ induce epimorphisms of the respective cohomology rings. Then the embedding $\widetilde{D}\to \widetilde{X}$ induces an epimorphism of the respective cohomology rings.
\end{lm}
\begin{proof}
Let $E'=\P(\nu')$ be the exceptional divisor of the blow-up $\Bl_{Z\cap D} D\to D$, where $\nu$ and $\nu'$ are the normal bundles of the inclusions $Z\subset X$ and $Z\cap D\subset D$, respectively. The normal vector bundle $\nu'$ is a subbundle of $\nu|_{Z\cap D}$ due to the sequence $Z\cap D\to D\to X$ of embeddings. Consider the following commutative diagram
\begin{equation}\label{eq:blowupincl}
\begin{tikzcd}[column sep=5]
\tikzset{
  >/.tip={Stealth[length=3pt, width=4pt, inset=1.8pt]}
}
H^*(\widetilde{X}) \arrow{rr}{\widetilde{\iota}^*}\arrow{d} && H^*(\widetilde{D})\arrow{d}\\
\biggl(H^*(X)\oplus H^*(\P(\nu))\biggr)/H^*(Z) \arrow{rr} && \biggl(H^*(D)\oplus H^*(\P(\nu'))\biggr)/H^*(Z\cap D)\\
\end{tikzcd}
\end{equation}
where the vertical arrows are the isomorphisms from Proposition \ref{pr:blow-up}, and the lower arrow is induced by the embeddings $\iota$ and $\P(\nu')\to\P(\nu)$ (by naturality). By the condition of the lemma, $\iota^*\colon H^*(X)\to H^*(D)$ is epimorphic. By Lemma \ref{lm:emb} and the assumption, the composition $H^*(\P(\nu))\to H^*(\P(\nu|_{Z\cap D})) \to H^*(\P(\nu'))$, induced by the natural embeddings, is epimorphic. Then the lower arrow in \eqref{eq:blowupincl} is epimorphic. By the commutativity of \eqref{eq:blowupincl} we conclude that $\widetilde{\iota}^*$ is epimorphic. This completes the proof.
\end{proof}

Let $Z_0\subset Z_1\subset\dots \subset Z_k$ be any closed connected submanifolds of the complex manifold $X$. Denote by $\widetilde{Z}_j$ the strict transform of the subvariety $Z_j$ under the blow-up $\widetilde{X}=\Bl_{Z_0} X\to X$ of $X$ along $Z_{0}$, where $j=1,\dots,k$. We generalise Lemma \ref{lm:blow_div} as follows.

\begin{lm}\label{lm:constrepi}
$(i)$ Assume that $Z_j$ and $D$ intersect transversally in $X$ for any $j=0,\dots,k$. Then $\widetilde{Z}_j$ and $\widetilde{D}$ intersect transversally in $\widetilde{X}$, where $j=1,\dots,k$;

\smallskip
\noindent $(ii)$ In addition to the condition $(i)$, suppose that the embeddings $D\to X$ and $Z_j\cap D\to Z_j$ induce epimorphisms of the respective cohomology rings for any $j=0,\dots,k$. Then the embeddings $\widetilde{D}\to \widetilde{X}$ and $\widetilde{Z}_j\cap \widetilde{D}\to \widetilde{Z}_j$ induce epimorphisms of the respective cohomology rings for any $j=1,\dots,k$.
\end{lm}
\begin{proof}
The claim $(i)$ follows from Proposition \ref{pr:blow-up} immediately. Now we prove $(ii)$. The claim about $\widetilde{D}\to \widetilde{X}$ follows by substituting $X,Z_0,D$ for $X,Z,D$ in Lemma \ref{lm:blow_div}. The claim about $\widetilde{Z}_j\cap \widetilde{D}\to \widetilde{Z}_j$ follows by substituting $Z_i,Z_0,Z_i\cap D$ for $X,Z,D$ in Lemma \ref{lm:blow_div}.
\end{proof}

See \S \ref{sec:def} for the definitions of $f_{i,j}$, $r_{i,j}$.

\begin{thm}\label{thm:cohann}
$(1)$ The embedding $f_{i,j}\colon BR_{i,j}\to BF_{i}\times \P^{j}$ induces epimorphism in cohomology. One has the ring isomorphism
\[
H^*(BR_{i,j})\simeq \frac{\Z[x_1,\dots,x_i,y]}{(x_q^2-x_q x_{q-1},\ y^{j+1}|\ q=1,\dots,i)}\bigg/\Ann (x_i+y),
\]
where $x_0:=0$.

\medskip
\noindent $(2)$ The embedding $r_{i,j}\colon R_{i,j}\to BF_{i}\times BF_{j}$ induces epimorphism in cohomology. One has the ring isomorphism
\[
H^*(R_{i,j})\simeq \frac{\Z[x_1,\dots,x_i,y_1,\dots,y_j]}{(x_q^2-x_q x_{q-1},\ y_r^2-y_r y_{r-1}|\ q=1,\dots,i;\ r=1,\dots, j)}\bigg/\Ann (x_i+y_j),
\]
where $x_0:=0,y_0:=0$.
\end{thm}
\begin{proof}
Propositions \ref{pr:brijdef}, \ref{pr:sijdef} and Lemma \ref{lm:constrepi} imply that $f_{i,j}^*$, $r_{i,j}^*$ are epimorphic. The respective kernels are given in Proposition \ref{pr:kerdual}. It remains to compute the cohomology of the respective Cartesian products. This follows by K\"{u}nneth formula from the computation of the cohomology rings of $\P^n$, $BF_n$ (see \eqref{eq:bfncohom}).
\end{proof}

\begin{ex}\label{ex:cohblowr22}
By Theorem \ref{thm:1blow}, $R_{2,2}$ is the blow-up of $BF_{1}\times BF_{2}$ along $R_{1,1}$. The normal bundle of the composition $R_{1,1}\to BF_{1}\times BF_{1}\to BF_{1}\times BF_{2}$ of embeddings is the restriction $\nu$ of $\beta_1^{\vee}\otimes(\beta'_1)^{\vee}\oplus (\beta'_1)^{\vee}$ to $R_{1,1}$. The irreducible rational curve $R_{1,1}$ is obtained by taking subsequently the divisors corresponding to the algebraic line bundles $\beta_1^{\vee}\otimes(\beta'_1)^{\vee}$, $(\beta'_2)^{\vee}$ over $BF_{1}\times BF_{2}$. Hence, $\omega_{R_{1,1}}=(x_1+y_1)y_2$. Clearly, $H^*(R_{1,1}\mid \Z)\simeq \Z[t]/(t^2)$, where $\eta\to \P^1$ is the tautological line bundle and $t=c_1(\eta^{\vee})$. It is not hard to compute the Chern class $c(\nu)$ to be $1+3t$ of $\nu$. Hence, by Proposition \ref{pr:blow-up}, one has
\begin{multline}\label{eq:r22}
H^*(R_{2,2};\Z)\simeq\frac{\bigl(\Z[x_1,y_1,y_2]/(x_1^2,y_1^2, y_2^2-y_1 y_2)\bigr)\oplus(\Z[t]/(t^2))\la v,v^2\ra}{\bigl(v^2+3vt+(x_1+y_1)y_2, (y_2-y_1)v, (y_2-x_1)v, tv-x_1v\bigr)}\simeq\\
\Z[x_1,y_1,y_2,v]/\biggl(x_1^2,y_1^2, y_2^2-y_1 y_2,v^2+3vy_2+(x_1+y_1)y_2, (y_2-y_1)v, (y_2-x_1)v\biggr).
\end{multline}
Here we can vanish $t$ by expressing the additive generators $tv$ and $tv^2$ as $x_1v$ and $x_1v^2$, respectively.
\end{ex}

\begin{ex}
It is not hard to compute the ideal $\Ann (x_2+y_2)$ of the ring $H^{*}(BF_{2}\times BF_{2};\Z)$ to be
\[
\bigl((x_2-x_1)(y_2-y_1),x_2^2+x_2y_2+y_2^2\bigr).
\]
Hence, by Theorem \ref{thm:cohann}, one has
\begin{equation}\label{eq:s22ann}
H^*(R_{2,2};\Z)\simeq\Z[x_1,x_2,y_1,y_2]/\bigl(x_1^2,x_2^2-x_1 x_2,y_1^2,y_2^2-y_1 y_2,(x_2-x_1)(y_2-y_1),x_2^2+x_2y_2+y_2^2\bigr).
\end{equation}
The isomorphism
\[
\Z[x_1,x_2,y_1,y_2]\to\Z[x_1,y_1,y_2,v],\ (x_1,y_1,y_2,x_2)\mapsto x_1,x_1+v,y_1,y_2,
\]
of polynomial rings induces the isomorphism between the quotient rings, which are given on the right hand sides of \eqref{eq:s22ann} and \eqref{eq:r22}. A similar computation shows that
\[
H^*(BR_{3,2};\Z)\simeq\Z[x_1,x_2,x_3,y]/\bigl(x_1^2,x_2^2-x_1 x_2, x_3^2-x_2 x_3, y^3,x_2 y^2  - x_3 y^2 , x_1 x_3 y - x_3^2 y - x_1 y^2  + x_3 y^2 , x_3^3  - x_3^2 y + x_3 y^2\bigr).
\]
\end{ex}

\subsection{Betti numbers}

Consider the Hodge-Deligne polynomial $e(X)(u,v):=\sum_{i,j} h^{i,j}(X)u^i v^j$ of a quasiprojective complex algebraic variety $X$ (see \cite{da-kh-87}, \cite{gu-10}).

\begin{pr}[{\cite[p.929]{da-kh-87}}]\label{pr:hdmult}
$(i)$ For any quasiprojective complex algebraic varieties $Y\subseteq X$ one has
\[
e(X)(u,v)=e(Y)(u,v)+e(X\setminus Y)(u,v);
\]

\smallskip
\noindent $(ii)$ For any integer $n\geq 0$ one has $e(\P^n)(u,v)= 1+uv+\dots+(uv)^n$;

\smallskip
\noindent $(iii)$ For any algebraic $F$-bundle $E\to B$, where $B,F$ are nonsingular projective varieties, one has
\[
e(E)(u,v)=e(B)(u,v) e(F)(u,v).
\]

\smallskip
\noindent $(iv)$ For any closed immersion $Z\subset X$ of nonsingular projective algebraic varieties, the identity
\[
e(\Bl_{Z} X)(u,v)=e(X)(u,v) + (uv+\dots+(uv)^{k-1})e(Z)(u,v),
\]
holds, where $k$ is the complex codimension of \/ $Z\subset X$.
\end{pr}

For any complex projective manifold $X$, the $k$-th Betti number $b_{k}(X)$ of $X$ is equal to $\sum_{i+j=k} h^{i,j}(X)$ by the Hodge decomposition, where $k\geq 0$ is any integer. If $X$ has only diagonal Hodge numbers, i.e. $h^{i,j}(X)=0$ for any $i\neq j$, then we put $e(X)(t):=e(X)(u,v)$, where $t=uv$.

\begin{pr}\label{pr:hdpoly}
Let $i,j\geq 0$ be any integers. Then the following relations hold.

\begin{equation}\label{eq:hdp1}
e(BR_{i,j})(t)=(1+t)^{i}(1+t+\dots+t^{j-1}),\ \mbox{where }\ 0\leq i\leq j\ \mbox{and }\ 0<j;
\end{equation}
\begin{equation}\label{eq:hdp2}
e(BR_{i,j})(t)=(1+t)^{i}(1+t+\dots+t^{j-1})+t^j (1+t)^{i-j-1},\ \mbox{where }\ i> j> 0;
\end{equation}
\begin{equation}\label{eq:hdp3}
e(R_{i,j})(t)=(1+t)^{i+j-1}+t(1+t)^{i+j-3}+\dots+t^{i-1}(1+t)^{j-i+1}+t^{\min{\lb i,j\rb}} (1+t)^{i+j-2\min{\lb i,j\rb}-1},
\end{equation}
where $0<i,j$ and $i\neq j$;
\begin{equation}\label{eq:hdp4}
e(R_{i,i})(t)=(1+t)^{2i-1}+t(1+t)^{2i-3}+\dots+t^{i-1}(1+t),\ \mbox{where }\ 2\leq i.
\end{equation}
\end{pr}
\begin{proof}
By Proposition \ref{pr:toricbr} the variety $BR_{i,j}$ is the algebraic $\P^{j-1}$-bundle over $BF_{i}$ for any integers $i,j$ such that $0\leq i\leq j$ and $j>0$. The variety $BF_{i}$ is the tower of algebraic $\P^1$-bundles over the point. Hence, by Proposition \ref{pr:hdmult} one obtains the formula \eqref{eq:hdp1} from the Hodge-Deligne polynomial of the projective space.

We prove \eqref{eq:hdp2} by the induction on $j$. By Theorem \ref{thm:1blowbr} $(ii)$, the variety $BR_{i,1}$ is the blow-up of $BF_{i-1}\times\P^1$ along its subvariety $BF_{i-2}$. Hence, by Proposition \ref{pr:hdmult},
\[
e(BR_{i,1})(t)=(1+t)^{i-1}(1+t)+t(1+t)^{i-2}=(1+t)^{i}+t(1+t)^{i-2},
\]
which proves the induction basis $j=1$. Assume that \eqref{eq:hdp2} holds for $j=j_{0}-1\geq 1$. By Theorem \ref{thm:1blowbr} $(ii)$, the variety $BR_{i,j}$ is the blow-up of $BF_{i-1}\times \P^j$ along its subvariety $BR_{i-1,j-1}$. We conduct the computation for $j=j_{0}$ by using the induction hypothesis and Proposition \ref{pr:hdmult} as follows.
\begin{multline*}
e(BR_{i,j})(t)=(1+t)^{i-1}(1+t+\dots+t^{j})+t\biggl((1+t)^{i-1}(1+t+\dots+t^{j-2})+t^{j-1} (1+t)^{i-j-1}\biggr)=\\
(1+t)^{i}(1+t+\dots+t^{j-1})+t^{j} (1+t)^{i-j-1}.
\end{multline*}
This proves the identity \eqref{eq:hdp2}.

It is enough to prove \eqref{eq:hdp3} only for any integers $i,j\geq 0$ such that $i<j$, because $R_{i,j}\simeq R_{j,i}$. We prove \eqref{eq:hdp3} by the induction on $j$. For $j=1$, \eqref{eq:hdp3} follows from \eqref{eq:hdp2}, since $R_{1,j}\simeq R_{j,1}=BR_{j,1}$. Assume that \eqref{eq:hdp3} holds for $j=j_{0}-1$. Let $j=j_0$. By Theorem \ref{thm:1blow} $(ii)$, the variety $R_{i,j}$ is the blow-up of $BF_{i}\times BF_j$ along its subvariety $R_{i-1,j-1}$. We conduct the computation for $j=j_{0}$ by using the induction hypothesis and Proposition \ref{pr:hdmult} as follows.
\[
e(R_{i,j})=(1+t)^{i+j-1}+t\biggl((1+t)^{i+j-3}+t(1+t)^{i+j-5}+\dots+t^{i-1}(1+t)^{j-i-1}\biggr).
\]
This proves the identity \eqref{eq:hdp3}.

Finally, prove \eqref{eq:hdp4} by the induction on $j$. Note that $R_{2,2}$ is the blow-up of $BF_{1}\times BF_{2}$ along its subvariety $\P^1$. By Proposition \ref{pr:hdmult}, then one has the identity
\[
e(R_{2,2})(t)=(1+t)^3+t(1+t),
\]
which proves the induction basis $j=2$. Assume that \eqref{eq:hdp4} holds for $i=i_0-1$. By Theorem \ref{thm:1blow} $(ii)$, the variety $R_{i,i}$ is the blow-up of $BF_{i-1}\times BF_{i}$ along its subvariety $R_{i-1,i-1}$.We conduct the computation for $i=i_0$ by using the induction hypothesis and Proposition \ref{pr:hdmult} as follows.
\[
e(R_{i,i})=(1+t)^{2i-1}+t\biggl((1+t)^{2i-3}+t(1+t)^{2i-5}+\dots+t^{i-1}(1+t)\biggr).
\]
The proof is complete.
\end{proof}

\begin{cor}
Let $i,j,k\geq 0$ be any integers. Then one has the following formulas.
\[
b^{2k}(BR_{i,j})=\binom{i}{k}+\binom{i}{k-1}+\dots+\binom{i}{k-j+1}+\binom{i-j-1}{k-j},\mbox{ where } i>j>0;
\]
\[
b^{2k}(R_{i,j})=\binom{i+j-1}{k}+\binom{i+j-3}{k-1}+\dots+\binom{i+j-2\min{\lb i,j\rb}-1}{k-\min{\lb i,j\rb}},\mbox{ where } 0<i,j \mbox{ and } i\neq j;
\]
\[
b^{2k}(R_{i,i})=\binom{2i-1}{k}+\binom{2i-3}{k-1}+\dots+\binom{1}{k-i+1},\mbox{ where } 1<i.
\]
\end{cor}

\begin{rem}
The identities from Proposition \ref{pr:hdpoly} agree with the various algebraic fiber bundle structures on $BR_{i,j}$ and $R_{i,j}$ from Section \ref{sec:blowup} and the property of Hodge-Deligne polynomial from Proposition \ref{pr:hdmult} $(iii)$.
\end{rem}

\end{document}